\documentclass[english,reqno]{amsart}
\usepackage[T1]{fontenc}
\usepackage[latin9]{inputenc}
\usepackage{babel}
\usepackage{refstyle}
\usepackage{mathrsfs}
\usepackage{mathtools}
\usepackage{enumitem}
\usepackage{amstext}
\usepackage{amsthm}
\usepackage{amssymb}
\usepackage[all]{xy}
\PassOptionsToPackage{normalem}{ulem}
\usepackage{ulem}
\usepackage[unicode=true,pdfusetitle,
 bookmarks=true,bookmarksnumbered=false,bookmarksopen=false,
 breaklinks=false,pdfborder={0 0 1},backref=false,colorlinks=false]
 {hyperref}
\hypersetup{
 colorlinks=true,citecolor=blue,linkcolor=blue,linktocpage=true}

\makeatletter

\AtBeginDocument{\providecommand\remref[1]{\ref{rem:#1}}}
\AtBeginDocument{\providecommand\thmref[1]{\ref{thm:#1}}}
\AtBeginDocument{\providecommand\lemref[1]{\ref{lem:#1}}}
\AtBeginDocument{\providecommand\corref[1]{\ref{cor:#1}}}
\AtBeginDocument{\providecommand\exaref[1]{\ref{exa:#1}}}
\AtBeginDocument{\providecommand\defref[1]{\ref{def:#1}}}
\AtBeginDocument{\providecommand\secref[1]{\ref{sec:#1}}}
\AtBeginDocument{\providecommand\propref[1]{\ref{prop:#1}}}
\RS@ifundefined{subsecref}
  {\newref{subsec}{name = \RSsectxt}}
  {}
\RS@ifundefined{thmref}
  {\def\RSthmtxt{theorem~}\newref{thm}{name = \RSthmtxt}}
  {}
\RS@ifundefined{lemref}
  {\def\RSlemtxt{lemma~}\newref{lem}{name = \RSlemtxt}}
  {}

\numberwithin{equation}{section}
\numberwithin{figure}{section}
\numberwithin{table}{section}
\theoremstyle{plain}
\newtheorem{thm}{\protect\theoremname}[section]
  \theoremstyle{remark}
  \newtheorem{rem}[thm]{\protect\remarkname}
  \theoremstyle{definition}
  \newtheorem{defn}[thm]{\protect\definitionname}
  \theoremstyle{plain}
  \newtheorem{lem}[thm]{\protect\lemmaname}
  \theoremstyle{remark}
  \newtheorem*{claim*}{\protect\claimname}
  \theoremstyle{plain}
  \newtheorem{cor}[thm]{\protect\corollaryname}
  \theoremstyle{definition}
  \newtheorem{example}[thm]{\protect\examplename}
  \theoremstyle{plain}
  \newtheorem{prop}[thm]{\protect\propositionname}
  \theoremstyle{remark}
  \newtheorem*{acknowledgement*}{\protect\acknowledgementname}

\providecommand{\MR}[1]{}

\usepackage{bbm}
\allowdisplaybreaks
\usepackage{needspace}

\setlist[enumerate]{itemsep=5pt,topsep=3pt}
\setlist[enumerate,1]{label=\textup{(}\roman*\textup{)},ref=\roman*}
\setlist[enumerate,2]{label=\textup{(}\alph*\textup{)},ref=\theenumi \alph*}

\newref{lem}{refcmd={Lemma \ref{#1}}}
\newref{thm}{refcmd={Theorem \ref{#1}}}
\newref{cor}{refcmd={Corollary \ref{#1}}}
\newref{sec}{refcmd={Section \ref{#1}}}
\newref{subsec}{refcmd={Section \ref{#1}}}
\newref{chap}{refcmd={Chapter \ref{#1}}}
\newref{prop}{refcmd={Proposition \ref{#1}}}
\newref{exa}{refcmd={Example \ref{#1}}}
\newref{tab}{refcmd={Table \ref{#1}}}
\newref{rem}{refcmd={Remark \ref{#1}}}
\newref{def}{refcmd={Definition \ref{#1}}}
\newref{fig}{refcmd={Figure \ref{#1}}}
\newref{claim}{refcmd={Claim \ref{#1}}}

\AtBeginDocument{
  
}

\makeatother

  \providecommand{\acknowledgementname}{Acknowledgement}
  \providecommand{\claimname}{Claim}
  \providecommand{\corollaryname}{Corollary}
  \providecommand{\definitionname}{Definition}
  \providecommand{\examplename}{Example}
  \providecommand{\lemmaname}{Lemma}
  \providecommand{\propositionname}{Proposition}
  \providecommand{\remarkname}{Remark}
\providecommand{\theoremname}{Theorem}

\begin{document}

\title[Interpolation and reproducing kernels]{Reproducing kernels and choices of associated feature spaces, in
the form of $L^{2}$-spaces}

\author{Palle Jorgensen}

\address{(Palle E.T. Jorgensen) Department of Mathematics, The University
of Iowa, Iowa City, IA 52242-1419, U.S.A. }

\email{palle-jorgensen@uiowa.edu}

\urladdr{http://www.math.uiowa.edu/\textasciitilde{}jorgen/}

\author{Feng Tian}

\address{(Feng Tian) Department of Mathematics, Hampton University, Hampton,
VA 23668, U.S.A.}

\email{feng.tian@hamptonu.edu}
\begin{abstract}
Motivated by applications to the study of stochastic processes, we
introduce a new analysis of positive definite kernels $K$, their
reproducing kernel Hilbert spaces (RKHS), and an associated family
of feature spaces that may be chosen in the form $L^{2}\left(\mu\right)$;
and we study the question of which measures $\mu$ are right for a
particular kernel $K$. The answer to this depends on the particular
application at hand. Such applications are the focus of the separate
sections in the paper.
\end{abstract}



\maketitle
\tableofcontents{}

\section{Introduction}

The use of reproducing kernels and their reproducing kernel Hilbert
spaces (RKHSs) was initially motivated by problems in classical analysis,
and it was put into an especially attractive and useful form by Aronszajn
in the 1950ties. Since then the applications of kernel theory has
greatly expanded, both in pure and applied mathematics. An application
of more recent vintage is machine learning. The number of applied
areas include use of RKHSs in the study of stochastic processes, especially
as a tool in Ito calculus; and in machine learning (ML). The last
two are related, and they are the focus of our present paper. Dictated
by a number of practical applications of the theory of ML, starting
with a positive definite (p.d.) kernel $K$, it has proved useful
to study both the associated RKHS itself, as well as a variety of
choices of feature spaces (for details, see \remref{p1} inside the
paper); and the interplay between them.

Now motivated by related applications to the study of stochastic processes,
it is of special significance to focus on the cases when the family
of feature spaces may be chosen in the form $L^{2}\left(\mu\right)$;
but this then raises the question of which measures $\mu$ are right
for a particular kernel $K$, and its associated RKHS. The answer
to this depends on the particular application at hand. Such applications
are the focus of the separate sections below inside the paper. 

In our study of RKHSs and choices of feature spaces, we have focused
on those of especial relevance to analysis of Gaussian calculus, but
there are many others, for example, functional and harmonic analysis,
boundary value problems, PDE, geometry and geometric analysis, operator
algebras/theory, the theory of unitary representations, mathematical
physics, and the study of fractals and fractal measures. Even this
list is not exhaustive. Nonetheless, we have narrowed our scope, and
our choice of applications, for the present paper. The reader will
be able to follow up on the various other directions, not covered
here, with the use of our cited references, see especially our discussion
of the literature below.

\textbf{Discussion of the literature.} The theory of RKHS and their
applications is vast, and below we only make a selection. Readers
will be able to find more cited there. As for the general theory of
RKHS in the pointwise category, we find useful \cite{ABDdS93,AD92,AD93,MR2529882,MR3526117}.
The applications include fractals (see e.g., \cite{AJSV13,MR0008639});
probability theory \cite{MR3571702,MR3504608,MR3379106,MR0277027,MR3624688,MR3622604};
and learning theory \cite{MR2186447,MR2352607,MR2644077,MR3341221,MR3560092,MR3491114,MR3564937,MR3567444,MR3600372,MR2354721}.

\section{Reproducing kernels}

The present setting begins with a fixed positive definite (p.d.) kernel
$K$, i.e., a function $K:S\times S\longrightarrow\mathbb{R}$ where
$S$ is a set, and satisfying
\begin{equation}
\sum\nolimits _{i}\sum\nolimits _{j}\alpha_{i}\alpha_{j}K\left(s_{i},s_{j}\right)\geq0\label{eq:a1}
\end{equation}
for all $\left\{ \alpha_{i}\right\} _{1}^{n}$, $\left\{ s_{i}\right\} _{1}^{n}$,
$\alpha_{i}\in\mathbb{R}$, $s_{i}\in S$, and $n\in\mathbb{N}$. 
\begin{rem}
Even though we shall state our definitions and results in the special
case of real valued functions, the complex case will result from our
present setting with only minor modifications. But in order to minimize
technical points, we have restricted the present discussion to the
\emph{real} case. 

The two more general settings are as follows: (i) complex; and (ii)
operator valued.
\begin{enumerate}
\item There the definition is as in (\ref{eq:a1}), but now $K:S\times S\longrightarrow\mathbb{C}$,
and the p.d. assumption is instead:
\[
\sum\nolimits _{i}\sum\nolimits _{j}\alpha_{i}\overline{\alpha}_{j}K\left(s_{i},s_{j}\right)\geq0
\]
for all choices of $\left\{ \alpha_{i}\right\} $ and $\left\{ s_{i}\right\} $,
$\alpha_{i}\in\mathbb{C}$, $s_{i}\in S$, $1\leq i\leq n$. 
\item Let $H$ be a complex Hilbert space, and let $\mathscr{B}\left(H\right)=$
the algebra of all bounded linear operators in $H$, i.e., $H\longrightarrow H$.
In this case, our setting for the kernel $K$ is: $K:S\times S\longrightarrow\mathscr{B}\left(H\right)$,
and now we assume instead that, for all $s_{i}$, $h_{i}$, with $s_{i}\in S$,
$h_{i}\in H$, and $1\leq i\leq n$; and all $n\in\mathbb{N}$, we
have: 
\[
\sum\nolimits _{i}\sum\nolimits _{j}\left\langle K\left(s_{i},s_{j}\right)h_{i},h_{j}\right\rangle _{H}\geq0.
\]
\end{enumerate}
\end{rem}

\begin{defn}
\label{def:fs}Given a positive definite (p.d.) kernel $K$ on $S$,
we shall consider pairs $\left(F,\mathbf{H}\right)$ where $\mathbf{H}$
is a Hilbert space, and $F:S\longrightarrow\mathbf{H}$ is a function
satisfying 
\[
\left\langle F\left(s\right),F\left(t\right)\right\rangle _{\mathbf{H}}=K\left(s,t\right),\;\forall s,t\in S.
\]
If $\left(F,\mathbf{H}\right)$ satisfies this, we say that $\mathbf{H}$
is a \emph{feature space}, or a feature Hilbert space. 
\end{defn}

\begin{rem}
In a general setup, reproducing kernel Hilbert spaces (RKHSs) were
pioneered by Aronszajn in the 1950s \cite{MR0008639,MR0051437}; and
subsequently they have been used in a host of applications. The key
idea of Aronszajn is that a RKHS is a Hilbert space $\mathscr{H}\left(K\right)$
of functions $f$ on a set such that the values $f(x)$ are \textquotedblleft reproduced\textquotedblright{}
from $f$ and a vector $K_{x}$ in $\mathscr{H}\left(K\right)$, in
such a way that the inner product $\langle K_{x},K_{y}\rangle=:K\left(x,y\right)$
is a positive definite kernel. 

By a theorem of Kolmogorov, every Hilbert space may be realized as
a (Gaussian) reproducing kernel Hilbert space (RKHS), see e.g., \cite{PaSc75,IM65},
and the details below.
\end{rem}

Let $\left(\Omega,\mathscr{C},\mathbb{P}\right)$ be a probability
space. We will be interested in centered Gaussian processes $\left(X_{s}\right)_{s\in S}$
(see e.g., \cite{MR3559026,MR0133175}), indexed by $S$, satisfying
\begin{enumerate}
\item \label{enu:a1}$X_{s}$ is Gaussian w.r.t. a probability space $\left(\Omega,\mathscr{C},\mathbb{P}\right)$,
\item \label{enu:a2}$X_{s}\in L^{2}\left(\Omega,\mathbb{P}\right)$, and
\begin{align}
\mathbb{E}\left(X_{s}\right) & =0,\label{eq:a2}\\
\mathbb{E}\left(X_{s}X_{t}\right) & =K\left(s,t\right),\;\forall s,t\in S,
\end{align}
where $\mathbb{E}$ denotes the expectation with respect to $\mathbb{P}$. 
\end{enumerate}
Given a p.d. kernel $K$, it is well known that a Gaussian realization
as in (\ref{enu:a1})-(\ref{enu:a2}) always exists; in fact, we may
choose $\mathbb{P}$ such that $\Omega=\mathbb{R}^{S}=$ all functions
on $S$, $\mathscr{C}=$ the corresponding cylinder $\sigma$-algebra
of subsets of $\Omega$; and 
\begin{equation}
X_{s}\left(\omega\right):=\omega\left(s\right),\;\forall\omega\in\Omega,\:s\in S.\label{eq:a3}
\end{equation}

A p.d. kernel of particular interest in the present paper will be
as follows: 

Let $\left(V,\mathscr{B},\mu\right)$ be a measure space, where $\mu$
is assumed positive and \emph{$\sigma$-finite}. Set $\mathscr{B}_{fin}:=\left\{ A\in\mathscr{B}\mathrel{;}\mu<\infty\right\} $.
On $\mathscr{B}_{fin}\times\mathscr{B}_{fin}$, then define $K^{\left(\mu\right)}$
by 
\begin{equation}
K^{\left(\mu\right)}\left(A,B\right):=\mu\left(A\cap B\right),\;\forall A,B\in\mathscr{B}_{fin}.\label{eq:a4}
\end{equation}

It is immediate that $K^{\left(\mu\right)}$ is p.d., and there is
therefore a canonical associated centered Gaussian process $X=X^{\left(\mu\right)}$,
indexed by $\mathscr{B}_{fin}$, satisfying 
\begin{equation}
\mathbb{E}(X_{A}^{\left(\mu\right)}X_{B}^{\left(\mu\right)})=\mu\left(A\cap B\right),\;\forall A,B\in\mathscr{B}_{fin}.\label{eq:a5}
\end{equation}
We shall study this process in detail and show that it may be used
to interpolate any Markov process built on $\left(V,\mathscr{B}\right)$;
see \thmref{m1}. 

A tool in our analysis will be reproducing kernel Hilbert spaces (RKHSs).
Recall that every p.d. kernel $K$ has an associated and unique RKHS
$\mathscr{H}\left(K\right)$. The reproducing axiom is as follows:
$K\left(\cdot,s\right)\in\mathscr{H}\left(K\right)$, and 
\begin{equation}
F\left(s\right)=\left\langle F,K\left(\cdot,s\right)\right\rangle _{\mathscr{H}\left(K\right)},\;\forall s\in S,\:\text{for \ensuremath{F\in\mathscr{H}\left(K\right)}}.\label{eq:a6}
\end{equation}

We now present two general lemmas, applied to any p.d. kernel $K:S\times S\longrightarrow\mathbb{C}$.
Let $\mathscr{H}\left(K\right)$ be the corresponding RKHS. 
\begin{lem}
\label{lem:rk1}A function $\psi$ on $S$ is in $\mathscr{H}\left(K\right)$
$\Longleftrightarrow$ $\exists C=C_{\psi}<\infty$, a finite constant,
s.t. for $\forall n\in\mathbb{N}$, $\forall\left\{ s_{i}\right\} _{1}^{n}$,
$s_{i}\in S$, $\forall\left\{ \alpha_{i}\right\} _{1}^{n}$, $\alpha_{i}\in\mathbb{C}$,
we have 
\begin{equation}
\left|\sum\nolimits _{i}\alpha_{i}\psi\left(s_{i}\right)\right|^{2}\leq C_{\psi}\sum\nolimits _{i}\sum\nolimits _{j}\alpha_{i}\overline{\alpha}_{j}K\left(s_{i},s_{j}\right).\label{eq:rk1-1}
\end{equation}
\end{lem}

\begin{proof}
This is standard Aronszajn theory \cite{MR3526117,MR0051437,MR0088706}.
If (\ref{eq:rk1-1}) holds, then we may take the constant $C_{\psi}=\left\Vert \psi\right\Vert _{\mathscr{H}}^{2}$,
and it is the smallest choice of admissible constant. 
\end{proof}
\begin{lem}[Two kernels]
\label{lem:rk2} Let $K_{1}$ and $K_{2}:$ $S\times S\longrightarrow\mathbb{C}$
both be p.d.; and let $\mathscr{H}\left(K_{i}\right)$, $i=1,2$,
be the corresponding RKHSs. Then the following are equivalent:
\begin{enumerate}
\item $\mathscr{H}\left(K_{1}\right)\subseteq\mathscr{H}\left(K_{2}\right)$;
\item $\exists C<\infty$ such that for $\forall n\in\mathbb{N}$, $\forall\left\{ s_{i}\right\} _{1}^{n}$,
$\left\{ \alpha_{i}\right\} _{1}^{n}$, $s_{i}\in S$, $\alpha_{i}\in\mathbb{C}$,
we have the estimate:
\[
\sum\nolimits _{i}\sum\nolimits _{j}\alpha_{i}\overline{\alpha}_{j}K_{1}\left(s_{i},s_{j}\right)\leq C\sum\nolimits _{i}\sum\nolimits _{j}\alpha_{i}\overline{\alpha}_{j}K_{2}\left(s_{i},s_{j}\right).
\]
Stated equivalently, $CK_{2}-K_{1}$ is positive definite (p.d.). 
\end{enumerate}
\end{lem}

\begin{proof}
This follows immediately from (\ref{eq:a6}) and \lemref{rk1}.
\end{proof}

\section{\label{sec:wn}Application to white noise analysis}

White noise analysis serves as a versatile framework for stochastic
and infinite-dimensional analysis, with a growing number of applications
to neighboring areas, probability, mathematical statistics, and quantum
physics. The setting is that of (Gaussian, continuous parameter) white
noise \textemdash{} a generalized random process indexed by elements
in a $\sigma$-algebra and with independent values at disjoint sets;
informally, we may view it as an infinite system of coordinates on
which to base an infinite-dimensional calculus. More precisely, the
starting point is the $L^{2}$-space of a white noise measure (Wiener
measure). A common approach makes use of a certain choice of a Gelfand
triples \cite{MR562914,App09}. Our approach is both entirely different,
and it is more general. The wider aim is an infinite-dimensional differential
calculus, and calculus of variation. 

Let $\left(V,\mathscr{B}\right)$ be a measure space, and let $\mu$
be a $\sigma$-finite measure on $\left(V,\mathscr{B}\right)$; then
define $K=K^{\left(\mu\right)}$ as follows: $K^{\left(\mu\right)}:\mathscr{B}_{fin}\times\mathscr{B}_{fin}\longrightarrow\mathbb{R}$,
\begin{equation}
K^{\left(\mu\right)}\left(A,B\right)=\mu\left(A\cap B\right),\;\forall A,B\in\mathscr{B}_{fin},\label{eq:w1}
\end{equation}
where $\mathscr{B}_{fin}=\left\{ A\in\mathscr{B}\mathrel{;}\mu\left(A\right)<\infty\right\} $. 
\begin{lem}
$K^{\left(\mu\right)}$ as in (\ref{eq:w1}) is positive definite. 
\end{lem}

\begin{proof}
We have 
\begin{align*}
\sum\nolimits _{i}\sum\nolimits _{j}\alpha_{i}\alpha_{j}K^{\left(\mu\right)}\left(A_{i},A_{j}\right) & =\left\Vert \sum\nolimits _{i}\alpha_{i}K^{\left(\mu\right)}\left(\cdot,A_{j}\right)\right\Vert _{\mathscr{H}(K^{\left(\mu\right)})}^{2}\\
 & =\int\left|\sum\nolimits _{i}\alpha_{i}1_{A_{i}}\right|^{2}d\mu\geq0,
\end{align*}
for $\forall\left\{ \alpha_{i}\right\} _{1}^{n}$, $\alpha_{i}\in\mathbb{R}$,
$\forall\left\{ A_{i}\right\} _{1}^{n}$, $A_{i}\in\mathscr{B}_{fin}$,
$\forall n\in\mathbb{N}$. 
\end{proof}
\begin{thm}
\label{thm:n1}Let $\left(V,\mathscr{B},\mu\right)$ be a measure
space, $\mu$ assumed $\sigma$-finite (positive). Let $K^{\left(\mu\right)}\left(A,B\right):=\mu\left(A\cap B\right)$,
$A,B\in\mathscr{B}_{fin}$, be the corresponding p.d. kernel; and
let $\mathscr{H}(K^{\left(\mu\right)})$ be the RKHS. Then

\begin{align}
\mathscr{H}(K^{\left(\mu\right)}) & =\big\{ F\text{ signed measures on }\left(V,\mathscr{B}\right)\;s.t.\;\label{eq:n1-1}\\
 & \qquad dF\ll d\mu\left(\text{abs. cont}\right)\:\text{with }\frac{dF}{d\mu}\in L^{2}\left(\mu\right)\big\};\;\text{and}\nonumber \\
\left\Vert F\right\Vert _{\mathscr{H}(K^{\left(\mu\right)})} & =\left\Vert \frac{dF}{d\mu}\right\Vert _{L^{2}\left(\mu\right)}.\label{eq:n1-2}
\end{align}

\end{thm}

\begin{proof}
We may use \lemref{rk1} to show that $F$, as in (\ref{eq:n1-1}),
is indeed in $\mathscr{H}(K^{\left(\mu\right)})$. 

Assume $F$ is as specified in (\ref{eq:n1-1}); and set $\varphi=\frac{dF}{d\mu}$
($\in L^{2}\left(\mu\right)$), which is the condition from (\ref{eq:n1-1})
on the Radon-Nikodym derivative. We will show that, if $n\in\mathbb{N}$,
$\left(A_{i}\right)_{1}^{n}$, $A_{i}\in\mathscr{B}_{fin}$ (i.e.,
$\mu\left(A_{i}\right)<\infty$), $\alpha_{i}\in\mathbb{R}$, then
\begin{equation}
\left|\sum\nolimits _{i}\alpha_{i}F\left(A_{i}\right)\right|^{2}\leq\left\Vert \varphi\right\Vert _{L^{2}\left(\mu\right)}^{2}\sum\nolimits _{i}\sum\nolimits _{j}\alpha_{i}\overline{\alpha}_{j}\underset{=\mu\left(A_{i}\cap A_{j}\right)}{\underbrace{K^{\left(\mu\right)}\left(A_{i},A_{j}\right)}}\label{eq:n1-3}
\end{equation}
and so we conclude that $F\in\mathscr{H}(K^{\left(\mu\right)})$,
with $\left\Vert F\right\Vert _{\mathscr{H}(K^{\left(\mu\right)})}\leq\left\Vert \varphi\right\Vert _{L^{2}\left(\mu\right)}$.
It is in fact ``$=$''. See below. 

We now give the verification of (\ref{eq:n1-3}): Let $n$, $\left(A_{i}\right)_{1}^{n}$,
$\left(\alpha_{i}\right)_{1}^{n}$, and $\varphi:=\frac{dF}{d\mu}$
$\in L^{2}\left(\mu\right)$ be as stated in (\ref{eq:n1-1}), and
the discussion above; then 
\begin{eqnarray*}
\text{LHS}_{\left(\ref{eq:n1-3}\right)} & = & \left|\sum\nolimits _{i}\alpha_{i}F\left(A_{i}\right)\right|^{2}=\left|\sum\nolimits _{i}\alpha_{i}\int_{A_{i}}\varphi d\mu\right|^{2}\\
 & = & \left|\int_{V}\varphi\cdot\sum\nolimits _{i}\alpha_{i}1_{A_{i}}d\mu\right|^{2}\\
 & \underset{\left(\text{Schwarz}\right)}{\leq} & \left\Vert \varphi\right\Vert _{L^{2}\left(\mu\right)}^{2}\left\Vert \sum\nolimits _{i}\alpha_{i}1_{A_{i}}\right\Vert _{L^{2}\left(\mu\right)}^{2}\\
 & = & \left\Vert \varphi\right\Vert _{L^{2}\left(\mu\right)}^{2}\sum\nolimits _{i}\sum\nolimits _{j}\alpha_{i}\alpha_{j}\mu\left(A_{i}\cap A_{j}\right)
\end{eqnarray*}
which is the desired conclusion (\ref{eq:n1-3}).
\begin{claim*}
Every $F\in\mathscr{H}(K^{\left(\mu\right)})$ is a $\sigma$-additive
signed measure, i.e., if $A=\cup_{i=1}^{\infty}A_{i}$, $A_{i}\cap A_{j}=\emptyset$,
$i\neq j$; (sets in $\mathscr{B}_{fin}$) then 
\begin{equation}
F\left(A\right)=\sum\nolimits _{i=1}^{\infty}F\left(A_{i}\right).\label{eq:n1-4}
\end{equation}
\end{claim*}
\begin{flushleft}
Proof of (\ref{eq:n1-4}). 
\begin{alignat*}{2}
\text{LHS}_{\left(\ref{eq:n1-4}\right)} & =F\left(A\right)=\left\langle F,\mu\left(\cdot\cap A\right)\right\rangle _{\mathscr{H}(K^{\left(\mu\right)})}, & \quad & \mu\left(\cdot\cap A\right)=K^{\left(\mu\right)}\left(\cdot,A\right),\\
 & =\left\langle F,\sum\nolimits _{i=1}^{\infty}\mu\left(\cdot\cap A_{i}\right)\right\rangle _{\mathscr{H}(K^{\left(\mu\right)})}, &  & \mu\left(\cdot\cap A_{i}\right)=K^{\left(\mu\right)}\left(\cdot,A_{i}\right),\\
 & =\sum\nolimits _{i=1}^{\infty}\left\langle F,K^{\left(\mu\right)}\left(\cdot,A_{i}\right)\right\rangle _{\mathscr{H}(K^{\left(\mu\right)})}\\
 & =\sum\nolimits _{i=1}^{\infty}F\left(A_{i}\right).
\end{alignat*}
We used $K^{\left(\mu\right)}\left(\cdot,A_{i}\right)\perp K^{\left(\mu\right)}\left(\cdot,A_{j}\right)$
for $i\neq j$. 
\par\end{flushleft}
\begin{claim*}
For $A,B\in\mathscr{B}_{fin}$, we have 
\begin{equation}
\frac{dK^{\left(\mu\right)}\left(\cdot,A\right)}{d\mu}=1_{A}=\text{the indicator function.}\label{eq:n1-5}
\end{equation}
\end{claim*}
\begin{flushleft}
Proof. For $A,B\in\mathscr{B}_{fin}$, we have 
\[
K^{\left(\mu\right)}\left(A,B\right)=\int_{B}1_{A}\left(x\right)d\mu\left(x\right)=\mu\left(B\cap A\right),
\]
and (\ref{eq:n1-5}) follows.
\par\end{flushleft}
\begin{claim*}
If $F\in\mathscr{H}(K^{\left(\mu\right)})$, then $dF\ll d\mu$ where
$dF$ is the signed measure in (\ref{eq:n1-4}).
\end{claim*}
\begin{flushleft}
Proof. We show that $\left[\mu\left(A\right)=0\right]$ $\Longrightarrow$
$\left[F\left(A\right)=0\right]$. From the reproducing property in
$\mathscr{H}(K^{\left(\mu\right)})$, we have:
\[
F\left(A\right)=\left\langle F,K^{\left(\mu\right)}\left(\cdot,A\right)\right\rangle _{\mathscr{H}(K^{\left(\mu\right)})}=\left\langle F,\mu\left(\cdot\cap A\right)\right\rangle _{\mathscr{H}(K^{\left(\mu\right)})};
\]
hence, $\mu\left(A\right)=0$ $\Longrightarrow$ $F\left(A\right)=0$,
since $\mu\left(A\right)=0$ $\Longrightarrow$ $K^{\left(\mu\right)}\left(\cdot,A\right)=0$,
and so $F\left(A\right)=\left\langle F,K^{\left(\mu\right)}\left(\cdot,A\right)\right\rangle _{\mathscr{H}(K^{\left(\mu\right)})}=0$. 
\par\end{flushleft}

The proof of \thmref{n1} is complete. 

\end{proof}
\begin{cor}
Let $\left(V,\mathscr{B},\mu\right)$ be a fixed $\sigma$-finite
measure space, and let $\mathscr{H}(K^{\left(\mu\right)})$ be the
RKHS of $K^{\left(\mu\right)}\left(A,B\right):=\mu\left(A\cap B\right)$,
$A,B\in\mathscr{B}_{fin}$ (see (\ref{eq:w1})). Define $W^{\left(\mu\right)}$
as an isometry: $W^{\left(\mu\right)}:L^{2}\left(\mu\right)\ni f\longmapsto fd\mu\in\mathscr{H}(K^{\left(\mu\right)})$,
where $W^{\left(\mu\right)}\left(f\right)=fd\mu$ is a signed measure
on $\left(V,\mathscr{B}\right)$; then 
\[
W^{\left(\mu\right)}:L^{2}\left(\mu\right)\simeq\mathscr{H}(K^{\left(\mu\right)})
\]
is an isometric isomorphism onto $\mathscr{H}(K^{\left(\mu\right)})$. 
\end{cor}

\begin{thm}
\label{thm:n2}Let $\left(V,\mathscr{B}\right)$ be a measure space,
$\mu$ a $\sigma$-finite measure on $\mathscr{B}$ , and set $\mathscr{B}_{fin}=\left\{ A\in\mathscr{B}\mathrel{;}\mu\left(A\right)<\infty\right\} $.
Let $K$ be a p.d. kernel on $\mathscr{B}_{fin}\times\mathscr{B}_{fin}$,
and let $\mathscr{H}\left(K\right)$ be the corresponding RKHS. Suppose
$\mathscr{H}\left(K\right)$ consists of signed measures; and set
\begin{equation}
\mathscr{H}\left(\mu\right):=\left\{ F\in\mathscr{H}\left(K\right)\mathrel{;}dF\ll d\mu,\;\text{and }\frac{dF}{d\mu}\in L^{2}\left(\mu\right)\right\} .\label{eq:n2-1}
\end{equation}
Then 
\begin{equation}
\mathscr{H}\left(\mu\right)\subseteq\mathscr{H}(K^{\left(\mu\right)}),\label{eq:n2-2}
\end{equation}
where 
\begin{equation}
K^{\left(\mu\right)}\left(A,B\right):=\mu\left(A\cap B\right),\;\forall A,B\in\mathscr{B}_{fin};\label{eq:n2-3}
\end{equation}
and therefore $\exists c\left(\mu\right)<\infty$ such that 
\begin{equation}
c\left(\mu\right)K^{\left(\mu\right)}-K\label{eq:n2-4}
\end{equation}
is positive definite. 
\end{thm}

\begin{proof}
Let $F\in\mathscr{H}\left(\mu\right)$, see (\ref{eq:n2-1}); and
set $\varphi^{\left(F\right)}:=\frac{dF}{d\mu}$. Let $n\in\mathbb{N}$,
$\left\{ A_{i}\right\} _{1}^{n}$, $A_{i}\in\mathscr{B}_{fin}$, $\left\{ \alpha_{i}\right\} _{1}^{n}$,
$\alpha_{i}\in\mathbb{R}$; then 
\begin{eqnarray*}
\left|\sum\nolimits _{i}\alpha_{i}F\left(A_{i}\right)\right|^{2} & = & \left|\sum\nolimits _{i}\alpha_{i}\int_{A_{i}}\varphi^{\left(F\right)}\left(x\right)d\mu\left(x\right)\right|^{2}\\
 & = & \left|\int_{V}\left(\sum\nolimits _{i}\alpha_{i}1_{A_{i}}\right)\varphi^{\left(F\right)}d\mu\right|^{2}\\
 & \underset{\left(\text{Schwarz}\right)}{\leq} & \sum\nolimits _{i}\sum\nolimits _{i}\alpha_{i}\alpha_{j}\mu\left(A_{i}\cap A_{j}\right)\Vert\varphi^{\left(F\right)}\Vert_{L^{2}\left(\mu\right)}^{2}.
\end{eqnarray*}
Hence by \lemref{rk1}, $F\in\mathscr{H}(K^{\left(\mu\right)})$,
see (\ref{eq:n2-3}); and $\left\Vert F\right\Vert _{\mathscr{H}(K^{\left(\mu\right)})}\leq\Vert\varphi_{L^{2}\left(\mu\right)}^{\left(F\right)}\Vert$. 

Conclusion (\ref{eq:n2-2}) now follows. Finally conclusion (\ref{eq:n2-4})
is immediate from \lemref{rk2}.
\end{proof}
\begin{example}
\label{exa:bm1}If $V=[0,\infty)$, $\mathscr{B}=$ Borel $\sigma$-algebra,
$\mu=dx=\lambda$, Lebesgue measure, then $X^{\left(\mu\right)}=$
standard Brownian motion. 
\end{example}

\begin{proof}
Let $A=\left[0,t\right]$, $B=\left[0,s\right]$, $s,t\in[0,\infty)$,
then $X_{A}^{\left(\mu\right)}=W_{\left[0,t\right]}$, $X_{B}^{\left(\mu\right)}=W_{\left[0,s\right]}$
satisfying 
\begin{equation}
E(X_{A}^{\left(\mu\right)}X_{B}^{\left(\mu\right)})=\mathbb{E}\left(W_{\left[0,t\right]}W_{\left[0,s\right]}\right)=\lambda\left(\left[0,t\right]\cap\left[0,s\right]\right)=s\wedge t,\label{eq:bm-1}
\end{equation}
so the standard p.d. kernel which determines Brownian motion; and
$d\left[W_{0,t}\right]_{2}=dt$, and $\left[W_{0,t}\right]_{2}=t$,
referring to the quadratic variation, see also \corref{qa} and \lemref{qv}. 
\end{proof}
Recall that in general, if $K$ is a p.d. kernel, $\mathscr{H}\left(K\right)$
the RKHS, then whenever $F:S\longrightarrow\mathscr{H}$ is a function
from $S$ into a Hilbert space $\mathscr{H}$ s.t. $K\left(t,s\right)=\left\langle F\left(s\right),F\left(t\right)\right\rangle _{\mathscr{H}}$,
there is then a corresponding transform $L=L_{F}:\mathscr{H}\longrightarrow\mathscr{H}\left(K\right)$,
given by 
\[
\left(Lh\right)\left(t\right)=\left\langle h,F\left(t\right)\right\rangle _{\mathscr{H}},\;\forall t\in S,\;\forall h\in\mathscr{H}.
\]
In \exaref{bm1}, we may apply this to this to the kernel $K=K^{\left(\mu\right)}$,
$S=\mathscr{B}_{fin}$, $K^{\left(\mu\right)}\left(A,B\right):=\mu\left(A\cap B\right)$,
$A,B\in\mathscr{B}_{fin}$, and let $F\left(A\right):=X_{A}^{\left(\mu\right)}$,
$A\in\mathscr{B}_{fin}$. 
\begin{cor}[An explicit transform]
\label{cor:et} Let $\left(V,\mathscr{B},\mu\right)$ be as in \exaref{bm1}.
Let $\Omega:=\mathbb{R}^{\mathscr{B}_{fin}}$, and set $F:\mathscr{B}_{fin}\longrightarrow L^{2}\left(\Omega,\mathscr{C},\mathbb{P}\right)$,
\[
F\left(A\right):=X_{A}^{\left(\mu\right)},\;A\in\mathscr{B}_{fin},
\]
where $X_{A}^{\left(\mu\right)}\in L^{2}\left(\Omega,\mathbb{P}\right)$
is the centered Gaussian process with covariance kernel $\mathbb{E}(X_{A}^{\left(\mu\right)}X_{B}^{\left(\mu\right)})=\mu\left(A\cap B\right)$. 

Then the transform $L:L^{2}\left(\Omega,\mathbb{P}\right)\longrightarrow\mathscr{H}(K^{\left(\mu\right)})$
is 
\[
\left(Lh\right)\left(A\right)=\mathbb{E}(hX_{A}^{\left(\mu\right)}),\;\forall A\in\mathscr{B}_{fin},\;\forall h\in L^{2}\left(\Omega,\mathbb{P}\right).
\]

Let $\mathscr{F}_{V}:=$ all measurable functions in $\left(V,\mathscr{B}\right)$,
and $f\in L^{2}\left(\mu\right)\subset\mathscr{F}_{V}$ (real valued),
we get the Ito-integral 
\begin{equation}
\int_{V}fdX:=\lim\sum_{i}f\left(s_{i}\right)X_{A_{i}}^{\left(\mu\right)},
\end{equation}
where the limit is taken over all measurable partitions of $V$, mesh
$\rightarrow0$. Then 
\begin{equation}
\mathbb{E}\left(\left|\int_{V}fdX\right|^{2}\right)=\int\left|f\right|^{2}d\mu.\label{eq:bm-2}
\end{equation}
\end{cor}

\begin{proof}
(sketch) For all partitions $\left\{ A_{i}\right\} $ on $V$, $s_{i}\in A_{i}$,
the Ito-isometry (\ref{eq:bm-2}) follows from the approximation:
\[
\mathbb{E}\left(\left|\sum\nolimits _{i}f\left(s_{i}\right)X_{A_{i}}^{\left(\mu\right)}\right|^{2}\right)=\sum\nolimits _{i}\left|f\left(s_{i}\right)\right|^{2}\mu\left(A_{i}\right)\xrightarrow[\;\text{mesh}\left(\pi\right)\rightarrow0\;]{}\int_{V}\left|f\right|^{2}d\mu.
\]
\end{proof}
\begin{rem}
Using this version of Ito integral, we get the following conclusions. 

Given a fixed $\sigma$-finite measure space $\left(V,\mathscr{B},\mu\right)$,
set $K=K^{\left(\mu\right)}$, 
\[
K^{\left(\mu\right)}\left(A,B\right)=\mu\left(A\cap B\right),\;A,B\in\mathscr{B}_{fin},
\]
and let $L^{2}\left(\Omega,\mathscr{C},\mathbb{P}\right)$ be the
corresponding probability space s.t. 
\[
X_{A}^{\left(\mu\right)}\left(\omega\right)=\omega\left(A\right),\;\omega\in\Omega=\mathbb{R}^{\mathscr{B}}.
\]
Then 
\[
\mathbb{E}(X_{A}^{\left(\mu\right)}X_{B}^{\left(\mu\right)})=K^{\left(\mu\right)}\left(A,B\right)=\mu\left(A\cap B\right),
\]
and the Ito integral $X_{f}^{\left(\mu\right)}=\int_{V}fdX$ is well
defined with 
\[
\mathbb{E}(|X_{f}^{\left(\mu\right)}|^{2})=\int_{V}\left|f\right|^{2}d\mu,
\]
and 
\[
\mu=QV\left(X\right)=\left[X,X\right]=\left[X\right]_{2};
\]
see also \corref{qa} and \lemref{qv}. 

The correspondence $\{X_{A}^{\left(\mu\right)}\}_{A\in\mathscr{B}}\longleftrightarrow\{X_{f}^{\left(\mu\right)}\}_{f\in L^{2}\left(\mu\right)}$
is bijective. Easy direction: given $X_{f}^{\left(\mu\right)}$ as
above, $A\in\mathscr{B}$, set $f=1_{A}$. 

For details on Ito calculus and Brownian motion, see, e.g., \cite{MR2793121,MR2966130,MR3402823,MR2049045,MR0301806,MR562914}. 
\end{rem}

\begin{cor}
Given $\left(V,\mathscr{B},\mu\right)$ fixed, $\sigma$-finite measure
space, we introduce the kernel $K^{\left(\mu\right)}$, and the associated
centered Gaussian process $X:=X^{\left(\mu\right)}$. From our Ito-calculus,
it follows that $X^{\left(\mu\right)}$ may be realized in two equivalent
ways: 
\begin{enumerate}
\item \label{enu:br1}$\Omega=\mathbb{R}^{\mathscr{B}}=$ all functions
from $\mathscr{B}$ into $\mathbb{R}$, $X_{A}^{\left(\mu\right)}\left(\omega\right)=\omega\left(A\right)$,
$A\in\mathscr{B}$;
\item \label{enu:br2}$\Omega=\mathbb{R}^{V}=$ all functions from $V$
into $\mathbb{R}$, $X_{f}^{\left(\mu\right)}\left(\omega\right)=\omega\left(f\right)$,
$f\in L^{2}\left(\mu\right)$. 
\end{enumerate}
From standard Kolmogorov consistency theory \cite{MR562914,PaSc75},
in (\ref{enu:br1}) the probability measure $\mathbb{P}$ is defined
on the cylinder $\sigma$-algebra of $\mathbb{R}^{\mathscr{B}}$,
and in case (\ref{enu:br2}) it is defined on the $\sigma$-algebra
for $\mathbb{R}^{V}$. 

We also get two equivalent versions of the covariance function for
$X$, which is indexed by $\mathscr{B}$ or by $L^{2}\left(\mu\right)$: 
\begin{enumerate}[resume]
\item \begin{flushleft}
\label{enu:br3}$\mathbb{E}\left(X_{A}X_{B}\right)=\mu\left(A\cap B\right)$,
$A,B\in\mathscr{B}$, 
\par\end{flushleft}

\end{enumerate}
\noindent \begin{flushleft}
\qquad{}$\Updownarrow$
\par\end{flushleft}
\begin{enumerate}[resume]
\item \label{enu:br4}$\mathbb{E}\left(X_{f}X_{g}\right)=\int_{V}f\left(x\right)g\left(x\right)d\mu\left(x\right)=\mathbb{E}\left(\left(\int fdX\right)\left(\int gdX\right)\right)$,
$\forall f,g\in L^{2}\left(\mu\right)$, real valued, where $X_{f}=\int fdX$
is the Ito integral formula which made the link from (\ref{enu:br3})
to (\ref{enu:br4}).
\end{enumerate}
\end{cor}

\begin{cor}
Let $X:=X^{\left(\mu\right)}$ be the Gaussian process as above, and
\begin{equation}
X_{f}:=\int_{V}fdX,
\end{equation}
then 
\begin{equation}
\mathbb{E}\left(e^{iX_{f}}\right)=e^{-\frac{1}{2}\int_{V}\left|f\right|^{2}d\mu},\;\forall f\in L^{2}\left(\mu\right).\label{eq:br5}
\end{equation}
\end{cor}

\begin{proof}
Direct proof from the power series expansions. See, e.g., \cite{MR3441738,zbMATH06690858},
and the papers cited there.
\end{proof}
\begin{rem}
Note that (\ref{eq:br5}) is analogous to the Gelfand triple construction,
but more general. In the present setting, we do not need a Gelfand
triple in order to make process (\ref{eq:br5}) above. 
\end{rem}

\begin{cor}
If $\left\{ f_{n}\right\} _{n\in\mathbb{N}_{0}}$ is an orthonormal
basis (ONB) in $L^{2}\left(\mu\right)$, then $\left\{ X_{f_{n}}\right\} _{n\in\mathbb{N}_{0}}$
is an i.i.d. $N\left(0,1\right)$ system, i.e., $X_{f_{n}}=Z_{n}\sim N\left(0,1\right)$,
and the following Karhunen-Loeve decomposition holds: 
\begin{equation}
X_{A}=\sum_{n=0}^{\infty}\left(\int_{A}f_{n}d\mu\right)Z_{n},\;\forall A\in\mathscr{B}_{fin}.\label{eq:br6}
\end{equation}
\end{cor}

\begin{cor}
\label{cor:qa}Assume $\mu$ is non-atomic. Then the quadratic variation
of $X:=X^{\left(\mu\right)}$ is $\mu$ itself, i.e., if $B\in\mathscr{B}$,
$d\left[X,X\right]=d\mu$. 
\end{cor}

\begin{proof}
Let $B\in\mathscr{B}$ with a partition $\left\{ A_{i}\right\} $
s.t. $B=\cup A_{i}$, $A_{i}\cap A_{j}=\emptyset$, $i\neq j$. If
$\mu$ is non-atomic, then 
\begin{equation}
\underset{=:\left[X,X\right]=\left[X\right]_{2}}{\underbrace{\lim\sum_{i}\left(X_{A_{i}}\right)^{2}}}=\mu\left(B\right)1\label{eq:br7}
\end{equation}
where $1$ denotes the constant function in $L^{2}\left(\Omega,\mathbb{P}\right)$,
and the limit is over the set of all partitions of $B$ with mesh
tending to 0. See \lemref{qv} for additional details.
\end{proof}
\begin{cor}[Generalized Ito lemma]
Let $f:\mathbb{R}\longrightarrow\mathbb{R}$, or $\mathbb{C}$, $f\in C^{2}$,
then 
\begin{equation}
df\left(X_{s}\right)=f'\left(X_{s}\right)dX_{s}+\frac{1}{2}f''\left(X_{s}\right)d\mu\left(s\right),\label{eq:br8}
\end{equation}
or equivalently, 
\begin{equation}
f\left(X_{B}\right)=\int_{B}f'\left(X_{s}\right)dX_{s}+\frac{1}{2}\int_{B}f''\left(X_{s}\right)d\mu\left(s\right),\label{eq:br9}
\end{equation}
for $\forall B\in\mathscr{B}_{fin}$, where we used the Ito integral,
and $d\left[X,X\right]\left(s\right)=d\mu\left(s\right)$ for the
quadratic variation. 
\end{cor}

\begin{rem}
We can do most of the white noise analysis in the more general setting,
i.e., w.r.t (\ref{eq:br9}). 
\end{rem}

\begin{cor}
Let the setting be as in \thmref{n2}, i.e., $\left(V,\mathscr{B},\mu\right)$,
and $K$ are specified as in (\ref{eq:n2-1})-(\ref{eq:n2-2}). In
particular, in addition to $K$, we also have the $\mu$-kernel $K^{\left(\mu\right)}\left(A,B\right)=\mu\left(A\cap B\right)$
as in (\ref{eq:n2-3}). Let $X^{\left(K\right)}$ be the centered
Gaussian process with kernel $K$, i.e., 
\begin{equation}
\mathbb{E}(X_{A}^{\left(K\right)}X_{B}^{\left(K\right)})=K\left(A,B\right),\;\forall A,B\in\mathscr{B}_{fin}.\label{eq:br10}
\end{equation}
Let $X^{\left(\mu\right)}$ be the Gaussian process (\ref{eq:a5})
with Ito integral 
\begin{equation}
X_{f}^{\left(\mu\right)}=\int_{V}f\left(x\right)dX_{x}^{\left(\mu\right)},\;f\in L^{2}\left(\mu\right),\label{eq:br11}
\end{equation}
and 
\begin{equation}
\mathbb{E}\left(|X_{f}^{\left(\mu\right)}|^{2}\right)=\int_{V}\left|f\right|^{2}d\mu;\label{eq:br12}
\end{equation}
see \corref{et}. 

Then there is a function 
\begin{equation}
G:\mathscr{B}_{fin}\times V\longrightarrow\mathbb{R},\label{eq:br13}
\end{equation}
measurable in the second variable, such that
\begin{align}
G\left(A,\cdot\right) & \in L^{2}\left(\mu\right),\;\forall A\in\mathscr{B}_{fin};\label{eq:br14}\\
K\left(A,B\right) & =\int_{V}G\left(A,x\right)G\left(B,x\right)d\mu\left(x\right)\label{eq:br15}
\end{align}
(compare with \defref{p1} below), and 
\begin{equation}
X_{A}^{\left(K\right)}=\int_{V}G\left(A,x\right)dX_{x}^{\left(\mu\right)}.\label{eq:br16}
\end{equation}
\end{cor}

\begin{proof}
The existence of $G$ follows from \thmref{n2}, and the Hida-Cramer
transform \cite{MR3445605}. Hence, by (\ref{eq:br11}), we may define
a Gaussian process $X^{\left(K\right)}$ by (\ref{eq:br16}); and
for $A,B\in\mathscr{B}_{fin}$, we have
\begin{eqnarray*}
\mathbb{E}(X_{A}^{\left(K\right)}X_{B}^{\left(K\right)}) & = & K\left(A,B\right)\\
 & \underset{\left(\text{by }\left(\ref{eq:br12}\right)\right)}{=} & \int_{V}G\left(A,x\right)G\left(B,x\right)d\mu\left(x\right),
\end{eqnarray*}
which is the desired conclusion.
\end{proof}

\section{Gaussian interpolation of Markov processes}

Markov models, or hidden Markov models, are ubiquitous in model building,
e.g., to models for speech and handwriting recognition, to software,
and learning mechanisms in biological neural networks. Within the
study of support vector machines, one use of Markov processes is to
solve both the problem of classification, and that of clustering.
The list of optimization tasks includes that of maximizing an ``expected
goodness of classification,'' or a ``goodness of clustering'' criterion.
This in turn leads to the study of specific kinds of probability distribution
over sequences of vectors \textemdash{} for which we have good parameter
estimation, and good marginal distribution algorithms.

Hidden Markov models tend to be robust in many uses, for example,
in determining the nature of an input signal, given the corresponding
an output. The model aims to determine the most probable set of parameters
which dictate input states, when based on an observed sequence of
output states.

The literature is quite large: Here we mention just \cite{MR3451525,MR3617637,MR3646630},
and the papers cited there.

\subsection{The Markov processes}

In our previous work \cite{JoPT17b}, we already discussed applications
of the family of Gaussian processes from \secref{wn}. Our present
aim is to use them in an interpolation algorithm for non-atomic Markov
processes. 

Recall the Gaussian processes $\{X_{A}^{\left(\mu\right)}\mathrel{;}A\in\mathscr{B}_{fin}\}$,
such that 
\begin{equation}
\mathbb{E}(X_{A}^{\left(\mu\right)}X_{B}^{\left(\mu\right)})=\mu\left(A\cap B\right),\;\forall A,B\in\mathscr{B}_{fin};\label{eq:m1}
\end{equation}
where $\left(V,\mathscr{B},\mu\right)$ is a given measure space,
and $\mu$ assumed positive and $\sigma$-finite. 
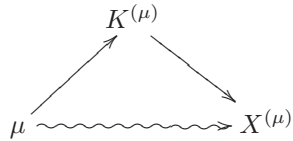
\[
\xymatrix{ & K^{\left(\mu\right)}\ar[rd]\\
\mu\ar@{~>}[rr]\ar[ur] &  & X^{\left(\mu\right)}
}
\]

Below we consider a family of Gaussian processes corresponding to
a given Markov process $P\left(x,A\right)$, where $\left(V,\mathscr{B}\right)$
is a measure space, $x\in V$, $P\left(x,\cdot\right)$ is a non-atomic
probability measure, i.e., $P\left(x,V\right)=1$. We shall denote
$P$ as the transition operator, defined for measurable functions
$f$ on $\left(V,\mathscr{B}\right)$, by
\begin{equation}
\left(Pf\right)\left(x\right)=\int_{V}f\left(y\right)P\left(x,dy\right),\;\forall x\in V.\label{eq:m2}
\end{equation}
Thus $P\left(\mathbbm{1}\right)=\mathbbm{1}$, and the constant function
$\mathbbm{1}$ is harmonic. (Also see \cite{JoPT17a,JoPT17c}, and
the papers cited there.)
\begin{lem}
Every generalized Markov process $P\left(x,\cdot\right)$ induces
a dual pairs of actions: 
\begin{enumerate}
\item action on measurable functions $f$ on $\left(V,\mathscr{B}\right)$,
\[
f\longmapsto\int f\left(y\right)P\left(x,dy\right)=\left(Pf\right)\left(x\right),\;x\in V;\;\text{and}
\]
\item action on signed measures $\nu$ on $\left(V,\mathscr{B}\right)$,
\[
\nu\longmapsto\int P\left(x,\cdot\right)d\nu\left(x\right)=P^{*}\left(\nu\right),
\]
where 
\[
\left(P^{*}\left(\nu\right)\right)\left(A\right)=\int P\left(x,A\right)d\nu\left(x\right),\;\forall A\in\mathscr{B}.
\]
\end{enumerate}
\end{lem}

As in standard Markov theory, 
\[
P_{2}\left(x,A\right)=\int P\left(x,dy\right)P\left(y,A\right)=P\left[P\left(\cdot,A\right)\right]\left(x\right),
\]
and inductively
\begin{equation}
P_{n+1}\left(x,A\right)=\int P_{n}\left(x,dy\right)P\left(y,A\right).\label{eq:m3}
\end{equation}

For each of the measures $P\left(x,\cdot\right),P_{2}\left(x,\cdot\right)\cdots,P_{n}\left(x,\cdot\right)$,
there is a corresponding white noise process $X^{\left(x\right)}$,
i.e., an indexed family of Gaussian processes $X_{A}^{\left(x\right)}$
$\sim$ $P\left(x,A\right)$, where $\mathbb{E}_{x}(X_{A}^{\left(x\right)})=0$,
and 
\begin{equation}
\mathbb{E}_{x}(X_{A}^{\left(x\right)}X_{B}^{\left(x\right)})=P\left(x,A\cap B\right),\;\forall A,B\in\mathscr{B}.\label{eq:m4}
\end{equation}
We now introduce a more general family of Ito integrals, and get a
new process $W_{A}^{\left(x\right)}$ which has $P_{2}\left(x,A\right)$
as its covariance kernel. See \thmref{m1} below. 
\begin{thm}[Interpolation]
\label{thm:m1} Let $P\left(x,\cdot\right)$ and $X_{A}^{\left(x\right)}$
be as specified above, see (\ref{eq:m1})-(\ref{eq:m4}). Set 
\begin{equation}
W_{A}^{\left(x\right)}:=\int_{V}X_{A}^{\left(y\right)}dX^{\left(x\right)}\left(y\right),\label{eq:m5}
\end{equation}
defined as an Ito integral. Then
\begin{enumerate}
\item $\{W_{A}^{\left(x\right)}\}$ is a Gaussian process;
\item $\mathbb{E}_{x}(W_{A}^{\left(x\right)})=0$;
\item $\mathbb{E}_{x}(W_{A}^{\left(x\right)}W_{B}^{\left(x\right)})=P_{2}\left(x,A\cap B\right)$,
$\forall A,B\in\mathscr{B}$, $\forall x\in V$. 
\item By induction, with an n-fold Ito integral from (\ref{eq:m5}), we
get $W_{A}^{\left(n\right)\left(x\right)}$ such that 
\begin{equation}
\mathbb{E}_{x}(W_{A}^{\left(n\right)\left(x\right)}W_{B}^{\left(n\right)\left(x\right)})=P_{n}\left(x,A\cap B\right),
\end{equation}
for $n\in\mathbb{N}$, $x\in V$, and $A,B\in\mathscr{B}$. 
\end{enumerate}
\end{thm}

\begin{proof}
Let $P\left(x,A\right)$ and $X_{A}^{\left(x\right)}$ be as specified
above. We then form the Ito integral $X_{\left(\cdot\right)}^{\left(x\right)}$
with $P\left(x,\cdot\right)$ as covariance. Note that for every $y\in V$,
$X^{\left(y\right)}$ is a centered Gaussian process with covariance
kernel 
\begin{equation}
\mathbb{E}_{y}(X_{A}^{\left(y\right)}X_{B}^{\left(y\right)})=P\left(y,A\cap B\right).\label{eq:m7}
\end{equation}
We shall show that $W_{\left(\cdot\right)}^{\left(x\right)}$ is also
a centered (i.e., mean zero) Gaussian process, now with $P_{2}\left(x,\cdot\right)$
as covariance measure, i.e., that $W^{\left(x\right)}$ from (\ref{eq:m5})
will satisfy:
\begin{equation}
\mathbb{E}_{x}(W_{A}^{\left(x\right)}W_{B}^{\left(x\right)})=P_{2}\left(x,A\cap B\right),\;\forall A,B\in\mathscr{B};
\end{equation}

The idea is that the white noise process interpolates the Markov process.
Aside from the induction, the key step in the argument is an analysis
of the Ito integral (\ref{eq:m5}). By general Ito theory, we have
\begin{align}
\mathbb{E}\left(|W_{A}^{\left(x\right)}|^{2}\right) & =\int\mathbb{E}\left(|W_{A}^{\left(y\right)}|^{2}\right)d[X^{\left(x\right)},X^{\left(x\right)}]\left(y\right)\label{eq:m9}\\
 & =\int P\left(y,A\right)P\left(x,dy\right)=P_{2}\left(x,A\right).\nonumber 
\end{align}
In the last step, we used (\ref{eq:m7}) on the term $\mathbb{E}(|X_{A}^{\left(y\right)}|^{2})$
in (\ref{eq:m9}); and used the formula for the quadratic variation
\begin{equation}
d[X^{\left(x\right)},X^{\left(x\right)}]\left(y\right)=P\left(x,dy\right).\label{eq:m10}
\end{equation}
See also \lemref{qv} below.
\end{proof}
Note that (\ref{eq:m10}) is a special case of an analogous property
of white noise, subject to a fixed measure $\mu$. Assume $\mu$ is
non-atomic, then 
\[
d[X^{\left(\mu\right)},X^{\left(\mu\right)}]\left(y\right)=d\mu\left(y\right),
\]
for $x$ fixed. We apply this to $d\mu\left(y\right)=P\left(x,dy\right)$,
and $X^{\left(x\right)}\sim P\left(x,\cdot\right)$.
\begin{lem}
\label{lem:qv}Give $\left(V,\mathscr{B},\mu\right)$ as in above,
and let $X:=X^{\left(\mu\right)}$ be the corresponding Gaussian process,
centered, with covariance given by (\ref{eq:m11}). Let $d\left[X,X\right]$
be the quadratic variation measure, i.e., for $B\in\mathscr{B}$,
$QV\left(B\right)=\lim\sum X_{A_{i}}^{2}$, where the limit is taken
over all measurable partitions $\pi=\left\{ A_{i}\right\} $ of $B$,
as $mesh\left(\pi\right)\rightarrow0$. Then $\mu\left(B\right)=QV\left(B\right)$,
and $d\mu=d\left[X,X\right]$. 
\end{lem}

\begin{proof}
It is true in general that if $\left(V,\mathscr{B},\mu\right)$ is
a $\sigma$-finite non-atomic measure space, and $X=X^{\left(\mu\right)}$
is the white noise Gaussian process determined by 
\begin{equation}
\mathbb{E}\left(X_{A}X_{B}\right)=\mu\left(A\cap B\right),\;\forall A,B\in\mathscr{B};\label{eq:m11}
\end{equation}
then for the quadratic variation measure $\left[X,X\right]=\left[X\right]_{2}$
, we have:
\[
\left[X\right]_{2}\left(B\right)=\mu\left(B\right),
\]
and so 
\[
d\left[X\right]_{2}\left(s\right)=d\mu\left(s\right).
\]

To see this, fix $B\in\mathscr{B}$, and take a limit on all measurable
partitions $\pi=\left\{ A_{i}\right\} $, where $A_{i}\cap A_{j}$,
$i\neq j$, $\cup A_{i}=B$, and set $\alpha_{i}=\mu\left(A_{i}\right)$.
A direct calculation gives 
\[
\mathbb{E}\left(X_{A_{i}}^{4}\right)=3\left(\mathbb{E}\left(X_{A_{i}}^{2}\right)\right)^{2}=3\alpha_{i}^{2},\;\text{and}
\]
\begin{eqnarray*}
 &  & \mathbb{E}\left(\left|\mu\left(B\right)-\sum\nolimits _{i}X_{A_{i}}^{2}\right|^{2}\right)\\
 & = & \mathbb{E}\left(\left|\sum\nolimits _{i}\mu\left(A_{i}\right)-\sum\nolimits _{i}X_{A_{i}}^{2}\right|^{2}\right)\\
 & = & \sum\nolimits _{i}\mathbb{E}\left(\left|\mu\left(A_{i}\right)-X_{A_{i}}^{2}\right|^{2}\right)+\sum\nolimits _{i\neq j}\mathbb{E}\left(\big(\mu\left(A_{i}\right)-X_{A_{i}}^{2}\big)\big(\mu(A_{j})-X_{A_{j}}^{2}\big)\right)\\
 & = & 2\sum\nolimits _{i}\mu\left(A_{i}\right)^{2}=2\sum\nolimits _{i}\alpha_{i}^{2}\xrightarrow[\;n\rightarrow\infty\;]{}0,
\end{eqnarray*}
since $\sum_{i}\alpha_{i}=\mu\left(B\right)>0$ is fixed. 
\end{proof}
In the proof of \lemref{qv}, we used the following fact:
\begin{lem}
Let $n\in\mathbb{N}$ be fixed, and let $\alpha_{i}>0$ satisfying
$\sum_{i}\alpha_{i}=1$, then 
\[
\inf_{n}\left\{ \sum\nolimits _{1}^{n}\alpha_{i}^{2}\mathrel{;}\sum\nolimits _{1}^{n}\alpha_{i}=1\right\} =0.
\]
\end{lem}

\begin{proof}
Fix $n\in\mathbb{N}$, and apply Lagrange multiplier to get $\alpha_{i}=1/n$,
for all $i$. Then 
\[
\sum_{1}^{n}\alpha_{i}^{2}=\frac{1}{n}\xrightarrow[\;n\rightarrow\infty\;]{}0.
\]
\end{proof}

\subsection{Fourier duality}
\begin{example}
Let $\left(V,\mathscr{B},\mu\right)=\left(\mathbb{R},\mathscr{B},\mu\right)$
where $\mu$ is a probability measure on $\mathbb{R}$. Let $X=X^{\left(\mu\right)}$
be the corresponding Gaussian process, and use the Ito integral to
define 
\begin{equation}
X_{e_{t}}=\int_{\mathbb{R}}e_{t}\left(x\right)dX_{x},\;\text{where}\quad e_{t}\left(x\right)=\exp\left(i2\pi xt\right),\;x,t\in\mathbb{R}.\label{eq:exq-1}
\end{equation}
Then 
\begin{equation}
\mathbb{E}\left(X_{e_{t}}\overline{X}_{e_{s}}\right)=\hat{\mu}\left(t-s\right),\;t,s\in\mathbb{R},\label{eq:exq-2}
\end{equation}
where $\hat{\mu}$ denotes the standard Fourier transform. 
\end{example}

\begin{proof}
Direct computation using (\ref{eq:exq-1}): 
\begin{equation}
\mathbb{E}\left(X_{e_{t}}\overline{X}_{e_{s}}\right)=\int e_{t}\left(x\right)\overline{e_{s}\left(x\right)}d\left[X\right]_{2}=\int e_{t-s}\left(x\right)d\mu\left(x\right)\label{eq:exq-3}
\end{equation}
where we used that $d\left[X\right]_{2}=d\mu$, see \lemref{qv}.
Then 
\begin{equation}
\int_{\mathbb{R}}\int_{\mathbb{R}}e_{t}\left(x\right)\overline{e_{s}\left(x\right)}\:\underset{\text{use orthogonality}}{\underbrace{\mathbb{E}\big(dX_{x}^{\left(\mu\right)}dX_{y}^{\left(\mu\right)}\big)}}=\int_{\mathbb{R}}e_{t}\left(x\right)\overline{e_{s}\left(x\right)}d\mu\left(x\right).\label{eq:exq-4}
\end{equation}
Note that if $A\cap B=\emptyset$, then $\mathbb{E}(X_{B}^{\left(\mu\right)}X_{B}^{\left(\mu\right)})=\mu\left(A\cap B\right)=\mu\left(\emptyset\right)=0$. 
\end{proof}
\begin{thm}[Duality]
The Fourier transform $X_{e_{t}}^{\left(\mu\right)}=\int e_{t}\left(x\right)dX_{x}^{\left(\mu\right)}$
is well defined, and it is a stationary process with covariance kernel
\begin{equation}
K^{F}\left(s,t\right)=\hat{\mu}\left(t-s\right),\label{eq:exq-5}
\end{equation}
where $\hat{\mu}$ is the standard Fourier transform of the measure
$\mu$.
\end{thm}

\begin{flushleft}
\textbf{Application.} In one of our earlier papers \cite{JoPT17b},
and papers cited there, we studied tempered measures $\mu$ on $\mathbb{R}$,
and processes $\left\{ Y_{\varphi}\right\} $ indexed by $\varphi\in\mathcal{S}$
(= the Schwartz space), and we get 
\begin{equation}
\mathbb{E}(\left|Y_{\varphi}\right|^{2})=\int_{\mathbb{R}}\left|\hat{\varphi}\left(x\right)\right|^{2}d\mu\left(x\right),\label{eq:exq-6}
\end{equation}
or equivalently, 
\begin{equation}
\mathbb{E}\left(e^{iY_{\varphi}}\right)=e^{-\frac{1}{2}\int\left|\hat{\varphi}\right|^{2}d\mu}.\label{eq:exq-7}
\end{equation}
But we can recover this setting from the case $\left(\mathbb{R},\mathscr{B},\mu\right)$
by setting 
\begin{equation}
Y_{\varphi}=\int\hat{\varphi}\left(x\right)dX_{x}^{\left(\mu\right)}\label{eq:exq-8}
\end{equation}
as an Ito integral. (Note that the RHS in (\ref{eq:exq-7}) is a continuous
positive definite function in $\varphi\in\mathcal{S}$ (the Schwartz
space), and so Minlos' theorem applies; see \cite{MR562914}.) Then
\begin{align*}
Y_{\varphi} & =\int\hat{\varphi}\left(x\right)dX_{x}^{\left(\mu\right)}=\iint\varphi\left(t\right)\overline{e_{t}\left(x\right)}dt\,dX_{x}^{\left(\mu\right)}\\
 & =\int\varphi\left(t\right)\left(\int\overline{e_{t}\left(x\right)}dX_{x}^{\left(\mu\right)}\right)dt,
\end{align*}
and 
\begin{align*}
\mathbb{E}(\left|Y_{\varphi}\right|^{2}) & =\mathbb{E}\left(\left|\int\hat{\varphi}\left(x\right)dX_{x}^{\left(\mu\right)}\right|^{2}\right)\\
 & =\mathbb{E}\left(\left|\int\varphi\left(t\right)dX_{e_{t}}^{\left(\mu\right)}dt\right|^{2}\right)\\
 & =\iint\varphi\left(t\right)\overline{\varphi\left(s\right)}\mathbb{E}\left(X_{e_{t}}\overline{X}_{e_{s}}\right)dtds\\
 & =\iint\varphi\left(t\right)\hat{\mu}\left(t-s\right)\overline{\varphi\left(s\right)}dtds\quad\text{by \ensuremath{\left(\ref{eq:exq-2}\right)}\&\ensuremath{\left(\ref{eq:exq-5}\right)}}\\
 & =\int_{\mathbb{R}}\left|\hat{\varphi}\left(x\right)\right|^{2}d\mu\left(x\right).
\end{align*}
This is the desired conclusion (\ref{eq:exq-6}). The idea is that
we get all these conclusions \emph{without} Gelfand triples. 
\par\end{flushleft}
\begin{rem}
The converse holds too. Suppose $\{Y_{\varphi}\}_{\varphi\in\mathcal{S}}$
is a Gaussian process (based on $\mathbb{R}$) computed from a tempered
measure $\mu$, 
\begin{equation}
\int_{\mathbb{R}}\frac{d\mu\left(x\right)}{1+x^{2}}<\infty,
\end{equation}
with the Gelfand triple 
\begin{equation}
\mathcal{S}\hookrightarrow L^{2}\left(\mathbb{R}\right)\hookrightarrow\mathcal{S}'\label{eq:exq-9}
\end{equation}
where $\mathcal{S}$ is the Schwartz space, and $\mathcal{S}'$ the
dual of all tempered distributions. Then $\{Y_{\varphi}\}_{\varphi\in\mathcal{S}}$
is the transform of the process $X^{\left(\mu\right)}$ $\Longleftrightarrow$
\begin{equation}
\mathbb{E}(\left|Y_{\varphi}\right|^{2})=\int_{\mathbb{R}}\left|\hat{\varphi}\left(x\right)\right|^{2}d\mu\left(x\right).\label{eq:exq-10}
\end{equation}
\end{rem}

\begin{proof}
Here the process $X^{\left(\mu\right)}$ is determined by measure
$\mu$, 
\begin{equation}
\begin{split}\mathbb{E}(X_{A}^{\left(\mu\right)}X_{B}^{\left(\mu\right)}) & =\mu\left(A\cap B\right),\quad\forall A,B\in\mathscr{B},\\
X_{f}^{\left(\mu\right)} & =\int fdX^{\left(\mu\right)}\quad\forall f\in L^{2}\left(\mathbb{R},\mu\right),\;\text{and}
\end{split}
\label{eq:exq-11}
\end{equation}
\begin{equation}
Y_{\varphi}=\int\hat{\varphi}\left(x\right)dX_{x}^{\left(\mu\right)}\quad\forall\varphi\in\mathcal{S}.\label{eq:exq-12}
\end{equation}
Indeed, we already proved that (\ref{eq:exq-12}) $\Longrightarrow$
(\ref{eq:exq-11}). Note that 
\begin{align*}
\mathbb{E}(\left|Y_{\varphi}\right|^{2}) & \underset{\left(\text{Ito}\right)}{=}\int\left|\hat{\varphi}\left(x\right)\right|^{2}d\mu\left(x\right)\\
 & =\iiint\varphi\left(x\right)\overline{\varphi\left(t\right)}e_{x}\left(s-t\right)d\mu\left(x\right)dsdt\\
 & =\iint\varphi\left(s\right)\overline{\varphi\left(t\right)}\hat{\mu}\left(s-t\right)dsdt.
\end{align*}
Now set $X_{e_{t}}=\int e_{t}\left(x\right)dX_{x}^{\left(\mu\right)}$,
then 
\begin{align*}
\mathbb{E}\left(X_{e_{t}}\overline{X}_{e_{s}}\right) & =\hat{\mu}\left(s-t\right),\;\text{and }\\
\int\hat{\varphi}\left(x\right)dX_{x}^{\left(\mu\right)} & =\int\varphi\left(t\right)X_{e_{t}}dt,\;\forall\varphi\in\mathcal{S}.
\end{align*}
\end{proof}

\subsection{A stochastic bilinear form}

Let $\left(V,\mathscr{B},\mu\right)$ be a fixed $\sigma$-finite
measure as above, and let $\left(\Omega,\mathscr{C},\mathbb{P}\right)$
be the measure space which realizes the process $\{X_{A}^{\left(\mu\right)}\}_{A\in\mathscr{B}}$,
and set $X_{f}^{\left(\mu\right)}=\int f\left(x\right)dX_{x}^{\left(\mu\right)}$
(Ito integral), $f\in L^{2}\left(\mu\right)$. Hence 
\begin{equation}
\mathbb{E}\left(|X_{f}^{\left(\mu\right)}|^{2}\right)=\int_{\Omega}|X_{f}^{\left(\mu\right)}|^{2}d\mathbb{P}=\int_{V}\left|f\right|^{2}d\mu\label{eq:tp1}
\end{equation}
holds by the generalized Ito isometry, and we define the transform
$L:H\longrightarrow\mathscr{H}\left(K^{\left(\mu\right)}\right)$,
\[
\left(Lh\right)\left(A\right)=\mathbb{E}(hX_{A}^{\left(\mu\right)}),\;\forall A\in\mathscr{B},
\]
where $H=L^{2}\left(\Omega,\mathbb{P}\right)$, and $\mathscr{H}\left(K^{\left(\mu\right)}\right)=$
the RKHS of $K^{\left(\mu\right)}\left(A,B\right):=\mu\left(A\cap B\right)$.
But using (\ref{eq:tp1}) again, we get 
\[
\mathbb{E}(hX_{A}^{\left(\mu\right)})=\int_{V}h1_{A}d\mu=\int_{A}hd\mu
\]
for $h\in L^{2}\left(\Omega,\mathbb{P}\right)$, and where $1_{A}$,
$\forall A\in\mathscr{B}$, is the indicator function on $A$, i.e.,
\[
1_{A}\left(x\right)=\delta_{x}\left(A\right)=\begin{cases}
1 & \text{if \ensuremath{x\in A}}\\
0 & \text{else}.
\end{cases}
\]

We have realized $\mathscr{H}\left(K^{\left(\mu\right)}\right)$ as
a Hilbert space of functions on $\left(V,\mathscr{B}\right)$ viz
$\simeq L^{2}\left(\mu\right)$. Note that since $K^{\left(\mu\right)}\left(A,B\right)=\mu\left(A\cap B\right)$
is a p.d. kernel on $\mathscr{B}_{fin}\times\mathscr{B}_{fin}$, initially
$\mathscr{H}\left(K^{\left(\mu\right)}\right)$ is a Hilbert space
of functions on $\mathscr{B}_{fin}$, but not functions on $\left(V,\mathscr{B}\right)$. 
\begin{cor}
Let $\left(V,\mathscr{B},\mu\right)$ satisfy the axioms from above,
let $K^{\left(\mu\right)}$ be the p.d. kernel on $\mathscr{B}_{fin}\times\mathscr{B}_{fin}$,
and $\mathscr{H}\left(K^{\left(\mu\right)}\right)$ be the corresponding
RKHS. Let $X_{f}^{\left(\mu\right)}$, $f\in L^{2}\left(\mu\right)$,
be the Gaussian process which extends $\{X_{A}^{\left(\mu\right)}\}_{A\in\mathscr{B}_{fin}}$,
and let $H=L^{2}\left(\Omega,\mathbb{P}\right)$ be the Gaussian Hilbert
space with inner product $\left\langle h_{1},h_{2}\right\rangle _{H}=\mathbb{E}\left(h_{1}h_{2}\right)$,
$\forall h_{i}\in H$. Then the bilinear mapping 
\[
L^{2}\left(\mu\right)\times H\ni\left(f,h\right)\longmapsto\mathbb{E}(X_{f}^{\left(\mu\right)}h)
\]
defines two operators in duality: 
\[
\xymatrix{L^{2}\left(\mu\right)\ar@/^{1.3pc}/[rr]^{I_{\mu}} &  & H\ar@/^{1.3pc}/[ll]^{\xi^{\left(\mu\right)}}}
\]
The Ito isometry $I_{\mu}:L^{2}\left(\mu\right)\longrightarrow H$,
\[
I_{\mu}\left(f\right)=X_{f}^{\left(\mu\right)},\;\forall f\in L^{2}\left(\mu\right),
\]
and the co-isometry $\xi^{\left(\mu\right)}:H\longrightarrow L^{2}\left(\mu\right)$,
determined by 
\[
\left\langle I_{\mu}\left(f\right),h\right\rangle _{H}=\langle f,\xi^{\left(\mu\right)}\left(h\right)\rangle_{L^{2}\left(\mu\right)},\;\forall f\in L^{2}\left(\mu\right),\:\forall h\in H.
\]
In particular, $I_{\mu}^{*}=\xi^{\left(\mu\right)}$, $\left(\xi^{\left(\mu\right)}\right)^{*}=I_{\mu}$. 
\end{cor}

\textbf{Duality.} Given $\left(V,\mathscr{B},\mu\right)$ for $\sigma$-finite
measure $\mu$, let $K^{\left(\mu\right)}\left(A,B\right):=\mu\left(A\cap B\right)$,
$A,B\in\mathscr{B}_{fin}$ where $\mathscr{B}_{fin}:=\left\{ A\in\mathscr{B}\mathrel{;}\mu\left(A\right)<\infty\right\} $.
Set 
\begin{equation}
\mathscr{H}(K^{\left(\mu\right)}):=\text{the RKHS of functions on \ensuremath{\mathscr{B}_{fin}}},\label{eq:du1}
\end{equation}
s.t. for all $F\in\mathscr{H}(K^{\left(\mu\right)})$, 
\[
F\left(A\right)=\langle F,K_{A}^{\left(\mu\right)}\rangle_{\mathscr{H}(K^{\left(\mu\right)})}=\left\langle F,\mu\left(A\cap\cdot\right)\right\rangle _{\mathscr{H}(K^{\left(\mu\right)})}.
\]
Set 
\begin{equation}
L\left(h\right)\left(A\right):=\langle h,X_{A}^{\left(\mu\right)}\rangle=\mathbb{E}(hX_{A}^{\left(\mu\right)}),\quad L^{*}(K_{A}^{\left(\mu\right)})=X_{A}^{\left(\mu\right)},\label{eq:du2}
\end{equation}
$h\in H:=L^{2}\left(\Omega,\mathscr{C},\mathbb{P}\right)$, $X_{A}^{\left(\mu\right)}=$
the Gaussian process of $K^{\left(\mu\right)}$, where 
\begin{equation}
X_{f}^{\left(\mu\right)}=\int f\left(x\right)dX_{x}^{\left(\mu\right)},\;f\in L^{2}\left(\mu\right).\label{eq:du3}
\end{equation}

A new operator $\xi:=\xi^{\left(\mu\right)}:H\longrightarrow L^{2}\left(\mu\right)$,
where $H:=L^{2}\left(\Omega,\mathscr{C},\mathbb{P}\right)$, and $L^{2}\left(\mu\right):=L^{2}\left(V,\mathscr{B},\mu\right)$. 
\begin{lem}
With the setting as above, $\xi^{\left(h\right)}\left(h\right)\in L^{2}\left(\mu\right)$,
$\forall h\in H$, is determined uniquely by
\begin{equation}
\int\xi^{\left(\mu\right)}\left(h\right)fd\mu=\mathbb{E}(hX_{f}^{\left(\mu\right)}),\;\forall h\in H,\:\forall f\in L^{2}\left(\mu\right).\label{eq:du4}
\end{equation}
\end{lem}

\begin{proof}
$\xi^{\left(h\right)}\left(h\right)$ is determined from (\ref{eq:du4})
and Reisz since 
\begin{equation}
\left|\mathbb{E}(hX_{f}^{\left(\mu\right)})\right|^{2}\leq\left\Vert h\right\Vert _{H}^{2}\int_{V}\left|f\right|^{2}d\mu.\label{eq:du5}
\end{equation}
So we only need to show the estimate (\ref{eq:du5}), but it follows
again from the Ito isometry, as follows: Let $h\in H$, and $f\in L^{2}\left(\mu\right)$,
then 
\begin{eqnarray*}
\left|\mathbb{E}(hX_{f}^{\left(\mu\right)})\right|^{2} & \underset{\left(\text{Schwarz}\right)}{\leq} & \left\Vert h\right\Vert _{L^{2}\left(\mathbb{P}\right)}^{2}\Vert X_{f}^{\left(\mu\right)}\Vert_{L^{2}\left(\mathbb{P}\right)}\\
 & = & \mathbb{E}(\left|h\right|^{2})\mathbb{E}\left(|X_{f}^{\left(\mu\right)}|^{2}\right)\\
 & = & \mathbb{E}(\left|h\right|^{2})\underset{=:\left\Vert f\right\Vert _{L^{2}\left(\mu\right)}^{2}}{\underbrace{\int_{V}\left|f\right|^{2}d\mu}}
\end{eqnarray*}
where we used the Ito isometry in the last step. 
\end{proof}
\begin{cor}
$\xi^{\left(\mu\right)}$ is contractive, 
\begin{equation}
\Vert\xi^{\left(\mu\right)}\left(h\right)\Vert_{L^{2}\left(\mu\right)}\leq\left\Vert h\right\Vert _{H},\;\forall h\in H.\label{eq:du6}
\end{equation}
\end{cor}

\begin{cor}
The two operators $I^{\left(\mu\right)}$ and $\xi^{\left(\mu\right)}$
are specified as follows: 
\begin{equation}
L^{2}\left(\mu\right)\ni\xymatrix{f\ar[rr]_{I^{\left(\mu\right)}} &  & X_{f}^{\left(\mu\right)}\ar@/_{1.3pc}/[ll]_{\left(\cdots\right)^{*}}}
\in H=L^{2}\left(\Omega,\mathscr{C},\mathbb{P}\right)\label{eq:du7}
\end{equation}
\[
L^{2}\left(\mu\right)\ni\xymatrix{\xi^{\left(\mu\right)}\left(h\right)\ar@/^{1.2pc}/[rr]^{\left(\xi^{\left(\mu\right)}\right)^{*}} &  & h\ar@/^{1.2pc}/[ll]^{\xi^{\left(\mu\right)}}\in H}
=L^{2}\left(\Omega,\mathscr{C},\mathbb{P}\right)
\]
We have 
\begin{equation}
\xi^{\left(\mu\right)}=\left(f\longrightarrow X_{f}^{\left(\mu\right)}\right)^{*}=\text{the adjoint operator, and}\label{eq:du8}
\end{equation}
\[
f\longrightarrow X_{f}^{\left(\mu\right)}=(\xi^{\left(\mu\right)})^{*},\;\text{and }
\]
\begin{equation}
f\longmapsto X_{f}^{\left(\mu\right)}=I^{\left(\mu\right)}\left(f\right)\text{ is isometric (Ito).}\label{eq:du9}
\end{equation}
Moreover, 
\begin{equation}
(I^{\left(\mu\right)})^{*}I^{\left(\mu\right)}=I_{L^{2}\left(\mu\right)},\;\text{while }\label{eq:du10}
\end{equation}
\begin{equation}
I^{\left(\mu\right)}(I^{\left(\mu\right)})^{*}=Q_{\mu}=\text{projection on \ensuremath{H};}\label{eq:du11}
\end{equation}
$\ensuremath{H=L^{2}\left(\Omega,\mathbb{P}\right)}$, $\ensuremath{Q=}\text{proj on }\ensuremath{R\left(I^{\left(\mu\right)}\right)}$. 

Or, we may rewrite (\ref{eq:du10})-(\ref{eq:du11}) as 
\begin{align}
\xi^{\left(\xi\right)}I^{\left(\mu\right)} & =\text{identity operator on \ensuremath{L^{2}\left(\mu\right)}}\\
I^{\left(\mu\right)}\xi^{\left(\mu\right)} & =\text{\ensuremath{Q_{u}} the proj in \ensuremath{H} onto the range of \ensuremath{I^{\left(\mu\right)}.}}
\end{align}
\end{cor}

\section{Parseval frames in the measure category}

Let $U$ be a set, and let $K:U\times U\longrightarrow\mathbb{R}$
be a positive definite (p.d.) kernel. We assume that the corresponding
RKHS $\mathscr{H}\left(K\right)$ is separable. (The result below
will apply \emph{mutatis mutandis} also to complex p.d. kernels $K:U\times U\longrightarrow\mathbb{C}$,
but for simplicity, we shall state our theorem only in the real case.) 

Let $\left(S,\mathscr{B},\mu\right)$ be a measure space with $\mu$
assumed positive and $\sigma$-finite. 
\begin{defn}
\label{def:p1}We shall say that $L^{2}\left(\mu\right)$ is a \emph{feature
space} if there is a function $r:U\longrightarrow L^{2}\left(\mu\right)$
such that 
\begin{equation}
K\left(x,y\right)=\int_{S}r_{x}\left(s\right)r_{y}\left(s\right)d\mu\left(s\right),\;\forall x,y\in U.\label{eq:p1}
\end{equation}
(In the complex case, the RHS in \ref{eq:p1} will instead be $\int_{S}r_{x}\left(s\right)\overline{r_{y}\left(s\right)}d\mu\left(s\right)$.
See also \defref{fs}.)
\end{defn}

\begin{rem}
\label{rem:p1}The notion of feature space derives from the setting
of machine learning \cite{MR1864085,MR2186447,MR2249315,MR2327597,MR2354721,MR3341221,MR3399224,MR3407331,MR3560092,MR2367342},
where learning optimization is made precise with the use of a choice
of reproducing kernel Hilbert space (RKHS). In practical terms, a
choice of feature space refers to a specified collections of features
that used to characterize data. For example, feature space might be
(Gender, Height, Weight, Age). In a support vector machine (SVM),
we might want to consider a different set of characteristics in order
to describe such data as (Gender, Height, Weight) etc; and we will
then arrive at mappings into other feature spaces. In finite dimension,
we may have feature spaces referring to some fixed number n of dimensions.
Since data is \textquotedblleft large\textquotedblright{} it is useful
to consider feature spaces to be function spaces, especially a choice
of $L^{2}$-spaces. The term feature space is used often in the machine
learning (ML) literature because a task in ML is feature extraction.
Hence we view all variables as features.
\end{rem}

\begin{lem}
We always have two distinguished feature spaces in the $L^{2}$-category:
$l^{2}\left(\mathbb{N}\right)$ vs Gaussian. 

CASE 1. $S=\mathbb{N}$ (note the separability assumption.), and $\mu:=$
the counting measure. If $\left\{ h_{n}\right\} _{n\in\mathbb{N}}$
is such that 
\begin{equation}
K\left(x,y\right)=\sum_{n=1}^{\infty}h_{n}\left(x\right)h_{n}\left(y\right),\;\forall x,y\in U;\label{eq:p2}
\end{equation}
then $h_{n}\in\mathscr{H}\left(K\right)$ for all $n\in\mathbb{N}$,
and $\left\{ h_{n}\right\} _{n\in\mathbb{N}}$ is a \uline{Parseval
frame} in $\mathscr{H}\left(K\right)$, i.e., 
\begin{equation}
\left\Vert F\right\Vert _{\mathscr{H}\left(K\right)}^{2}=\sum_{n=1}^{\infty}\left|\left\langle F,h_{n}\right\rangle _{\mathscr{H}\left(K\right)}\right|^{2};\label{eq:p3}
\end{equation}
and 
\begin{equation}
F\left(x\right)=\sum_{n=1}^{\infty}\left\langle F,h_{n}\right\rangle _{\mathscr{H}\left(K\right)}h_{n}\left(x\right)\label{eq:p4}
\end{equation}
is also strongly convergent, for all $x\in U$, and $F\in\mathscr{H}\left(K\right)$. 
\end{lem}

\begin{proof}
The lemma follows from standard RKHS theory \cite{MR0051437,MR0008639}.
An important point is to note that if $K$ and $\left\{ h_{n}\right\} _{n\in\mathbb{N}}$
satisfy the assumptions in CASE 1, then $h_{n}\in\mathscr{H}\left(K\right)$
for $\forall n\in\mathbb{N}$. We may get this as an application of
\lemref{rk1}. Indeed, let $n_{0}\in\mathbb{N}$ be fixed. Let $k\in\mathbb{N}$,
$\left\{ \alpha_{i}\right\} _{1}^{k}$, $\left\{ x_{i}\right\} _{1}^{k}$,
$\alpha_{i}\in\mathbb{R}$, $x_{i}\in U$; then 
\begin{eqnarray*}
\left|\sum\nolimits _{i}\alpha_{i}h_{n_{0}}\left(x_{i}\right)\right|^{2} & \leq & \sum\nolimits _{n\in\mathbb{N}}\left|\sum\nolimits _{i}\alpha_{i}h_{n}\left(x_{i}\right)\right|^{2}\\
 & = & \sum\nolimits _{i}\sum\nolimits _{j}\alpha_{i}\alpha_{j}\sum\nolimits _{n\in\mathbb{N}}h_{n}\left(x_{i}\right)h_{n}\left(x_{j}\right)\\
 & \underset{\text{by \ensuremath{\left(\ref{eq:p2}\right)}}}{=} & \sum\nolimits _{i}\sum\nolimits _{j}\alpha_{i}\alpha_{j}K\left(x_{i},x_{j}\right);
\end{eqnarray*}
and so the premise in \lemref{rk1} holds, and we conclude that $h_{n_{0}}\in\mathscr{H}\left(K\right)$. 

CASE 2. Set $S=\mathbb{R}^{U}$, $\mathscr{B}=$ the cylinder $\sigma$-algebra,
and $\mu=$ the Gaussian probability measure on $S$ determined by
its finite samples: $k\in\mathbb{N}$, $\left\{ x_{i}\right\} _{1}^{k}$,
$x_{i}\in U$. On $\mathbb{R}^{k}$, define the standard centered
Gaussian, so with mean 0, and covariance matrix $\left(K\left(x_{i},x_{j}\right)\right)_{i,j=1}^{k}$.
Then apply Kolmogorov consistency, and $\mu=\mathbb{P}_{Kolm}\left(K\right)$
will be the corresponding measure, also called the Wiener measure.
Setting, for $x\in U$, 
\begin{equation}
r_{x}\left(s\right)=s\left(x\right);\label{eq:p5}
\end{equation}
and the desired conclusions follow: 
\begin{enumerate}
\item Each $r_{x}\in L^{2}\left(\mu\right)$ is Gaussian,
\item $\mathbb{E}\left(r_{x}\right)=\int r_{x}d\mu=0$, 
\item $\mathbb{E}\left(r_{x}r_{y}\right)=K\left(x,y\right)$, $\forall x,y\in U$. 
\end{enumerate}
\end{proof}
\begin{thm}
Let $K:U\times U\longrightarrow\mathbb{R}$ be a positive definite
(p.d.) kernel, and let $\left(S,\mathscr{B},\mu,r\right)$ be a feature
space as specified in \defref{p1}; in particular, we have $r_{x}\in L^{2}\left(\mu\right)$
, $\forall x\in U$, and $K\left(x,y\right)=\int_{S}r_{x}\left(s\right)r_{y}\left(s\right)d\mu\left(s\right)$,
see (\ref{eq:p1}). 
\begin{enumerate}
\item Set 
\begin{equation}
R\left(x,A\right)=\int_{A}r_{x}\left(s\right)d\mu\left(s\right),\label{eq:p6}
\end{equation}
for $x\in U$, $A\in\mathscr{B}_{fin}$, where 
\begin{equation}
\mathscr{B}_{fin}=\left\{ A\in\mathscr{B}\mathrel{;}\mu\left(A\right)<\infty\right\} ;\label{eq:p7}
\end{equation}
then $R$ in (\ref{eq:p6}) is a measure in the second variable, and
it is measurable in $x$ $\left(\in U\right)$. 
\item For all $F\in\mathscr{H}\left(K\right)$, we have 
\begin{equation}
\left\Vert F\right\Vert _{\mathscr{H}\left(K\right)}^{2}=\int_{S}\left|\left\langle F\left(\cdot\right),R\left(\cdot,ds\right)\right\rangle _{\mathscr{H}\left(K\right)}\right|^{2},\label{eq:p8}
\end{equation}
and 
\begin{equation}
F\left(x\right)=\int_{S}r_{x}\left(s\right)\left\langle F\left(\cdot\right),R\left(\cdot,ds\right)\right\rangle _{\mathscr{H}\left(K\right)}.\label{eq:p9}
\end{equation}
\end{enumerate}
\end{thm}

\begin{rem}
We study the the parallel between the present conclusions (\ref{eq:p8})-(\ref{eq:p9}),
and the more familiar ones (\ref{eq:p3})-(\ref{eq:p4}) from standard
Parseval frame-theory, see, e.g., \cite{MR2254502,MR3204026,MR3275625},
and also see \cite{MR1101548,MR1240307,MR2468304} for direct integrals.
\end{rem}

\begin{proof}
Note first that there is a natural isometry $J$ defined by limits
and closure as follows: 
\begin{equation}
J\left(\sum\nolimits _{i}\alpha_{i}K\left(\cdot,x_{i}\right)\right)=\sum\nolimits _{i}\alpha_{i}r_{x_{i}}.\label{eq:p10}
\end{equation}
Indeed for finite sample $k$, $\alpha_{i}\in\mathbb{R}$, $x_{i}\in U$,
$1\leq i\leq k$, we have 
\begin{equation}
\left\Vert \sum\nolimits _{i}\alpha_{i}K\left(\cdot,x_{i}\right)\right\Vert _{\mathscr{H}\left(K\right)}^{2}=\left\Vert \sum\nolimits _{i}\alpha_{i}r_{x_{i}}\right\Vert _{L^{2}\left(\mu\right)}^{2}\label{eq:p11}
\end{equation}
since both sides in (\ref{eq:p11}) reduce to $\sum_{i}\sum_{j}\alpha_{i}\alpha_{j}K\left(x_{i},x_{j}\right)$.
As a result, in order to verify (\ref{eq:p8})-(\ref{eq:p9}), we
need only consider the case $F\left(\cdot\right)=K\left(\cdot,y\right)$
when $y\in U$ is fixed. Then it is enough to show that (\ref{eq:p8})
holds, and (\ref{eq:p9}) will follow. 

Let $y\in U$ be fixed, and assume $F\left(\cdot\right)=K\left(\cdot,y\right)$.
Then 
\begin{align*}
\text{LHS}_{\left(\ref{eq:p8}\right)} & =K\left(y,y\right),\;\text{and}\\
\text{\ensuremath{\text{RHS}_{\left(\ref{eq:p8}\right)}}} & =\int_{S}\left|r_{y}\left(s\right)\right|^{2}d\mu=K\left(y,y\right),
\end{align*}
by (\ref{eq:p1}). Similarly, 
\begin{align*}
\text{LHS}_{\left(\ref{eq:p9}\right)} & =K\left(x,y\right),\;\text{and}\\
\text{RHS}_{\left(\ref{eq:p9}\right)} & =\int_{S}r_{y}\left(s\right)r_{x}\left(s\right)d\mu\left(s\right)=K\left(x,y\right),
\end{align*}
again from an application of assumption (\ref{eq:p1}).
\end{proof}

\section{Transforms}

Let $K:U\times U\longrightarrow\mathbb{R}$ be a positive kernel.
(We shall state the result below for the real case but extensions
to p.d. functions with values in $\mathbb{C}$ are straightforward;
even to the case of operator valued kernels.) Now let $\left(S,\mathscr{B},\mu\right)$
be a measure space with $\mu$ fixed and assume $\sigma$-finite.
We shall further assume that $L^{2}\left(\mu\right)$ is a \emph{feature
space}; see \defref{p1} and \remref{p1}, i.e., we assume that there
is a function, $U\xrightarrow{\;F\;}L^{2}\left(\mu\right)$, $F\left(x\right)=r_{x}\left(\cdot\right)\left(\in L^{2}\left(\mu\right)\right)$
such that (\ref{eq:p1}) holds.
\begin{prop}
\label{prop:tr1}~
\begin{enumerate}
\item \label{enu:t1-1}With the setting $\left(K,U,\mu,\left\{ r_{x}\right\} _{x\in U}\right)$
as above, there is then a unique isometry $J:\mathscr{H}\left(K\right)\longrightarrow L^{2}\left(\mu\right)$
specified by 
\begin{equation}
J\left(K\left(\cdot,x\right)\right)=r_{x}.\label{eq:t1}
\end{equation}
\item \label{enu:t1-2}The adjoint operator of $J$ is $L:=J^{*}:L^{2}\left(\mu\right)\longrightarrow\mathscr{H}\left(K\right)$,
given by 
\begin{equation}
\left(Lh\right)\left(x\right)=\int_{S}h\left(s\right)r_{x}\left(s\right)d\mu\left(s\right),\label{eq:t2}
\end{equation}
$\forall h\in L^{2}\left(\mu\right)$, and $x\in U$. 
\item \label{enu:t1-3}$Q:=JJ^{*}=L^{*}L$ is the projection in $L^{2}\left(\mu\right)$
onto the closed subspace spanned by $\left\{ r_{x}\left(\cdot\right)\mathrel{;}x\in U\right\} \left(\subseteq L^{2}\left(\mu\right)\right)$. 
\end{enumerate}
\end{prop}

\begin{proof}
(\ref{enu:t1-1}) Let the setting be as specified in the Proposition.
Define $J$ as in (\ref{eq:t1}) and extend by linearity, first on
finite linear combinations 
\begin{equation}
J\left(\sum\nolimits _{i}\alpha_{i}K\left(\cdot,x_{i}\right)\right)=\sum\nolimits _{i}\alpha_{i}r_{x_{i}}\left(\cdot\right).\label{eq:t3}
\end{equation}
It is isometric, since 
\[
\left\Vert \sum\nolimits _{i}\alpha_{i}K\left(\cdot,x_{i}\right)\right\Vert _{\mathscr{H}\left(K\right)}^{2}=\sum\nolimits _{i}\sum\nolimits _{j}\alpha_{i}\alpha_{j}K\left(x_{i},x_{j}\right),
\]
and 
\begin{align*}
\int_{S}\left|\sum\nolimits _{i}\alpha_{i}r_{x_{i}}\left(\cdot\right)\right|^{2}d\mu & =\sum\nolimits _{i}\sum\nolimits _{j}\alpha_{i}\alpha_{j}\int_{S}r_{x_{i}}r_{x_{j}}d\mu\\
 & =\sum\nolimits _{i}\sum\nolimits _{j}\alpha_{i}\alpha_{j}K\left(x_{i},x_{j}\right),
\end{align*}
where we used (\ref{eq:p1}) in the last step. Since $J$ is thus
isometric on a dense subspace in $\mathscr{H}\left(K\right)$, it
extends uniquely by limits, to define an isometry $J:\mathscr{H}\left(K\right)\longrightarrow L^{2}\left(\mu\right)$
as required in (\ref{enu:t1-1}).

(\ref{enu:t1-2}) We now turn to the operator $L:L^{2}\left(\mu\right)\longrightarrow\mathscr{H}\left(K\right)$
as in (\ref{eq:t2}). The important point is that $L$ maps into $\mathscr{H}\left(K\right)$.
This follows from an application of \lemref{rk1} as follows. Let
$n\in\mathbb{N}$, $\alpha_{i}\in\mathbb{R}$, $x_{i}\in U$, $1\leq i\leq n$;
then 
\begin{eqnarray*}
\left|\sum\nolimits _{i}\alpha_{i}\left(Lh\right)\left(x_{i}\right)\right|^{2} & = & \left|\int_{S}\left(\sum\nolimits _{i}\alpha_{i}r_{x_{i}}\right)hd\mu\right|^{2}\\
 & \underset{\text{Schwarz}}{\leq} & \int_{S}\left|\sum\nolimits _{i}\alpha_{i}r_{x_{i}}\right|^{2}d\mu\int_{S}\left|h\right|^{2}d\mu\\
 & \underset{\text{by \ensuremath{\left(\ref{enu:t1-2}\right)}}}{=} & \sum\nolimits _{i}\sum\nolimits _{j}\alpha_{i}\alpha_{j}K\left(x_{i},x_{j}\right)\left\Vert h\right\Vert _{L^{2}\left(\mu\right)}^{2},
\end{eqnarray*}
and the conclusion now follows from \lemref{rk1}.

(\ref{enu:t1-3}) We have, since $Lh\in\mathscr{H}\left(K\right)$,
\begin{eqnarray*}
\left\langle K\left(\cdot,x\right),Lh\right\rangle _{\mathscr{H}\left(K\right)} & = & \left(Lh\right)\left(x\right)\\
 & \underset{\text{by \ensuremath{\left(\ref{eq:t2}\right)}}}{=} & \int_{S}r_{x}hd\mu=\left\langle r_{x},h\right\rangle _{L^{2}\left(\mu\right)},
\end{eqnarray*}
so $L^{*}\left(K\left(\cdot,x\right)\right)=r_{x}$, $L^{*}=J$, and
$J^{*}=L$, now follow from (\ref{enu:t1-1}). Since $J$ is isometric,
$Q=JJ^{*}$ is the projection specified in (\ref{enu:t1-3}) in the
Proposition.
\end{proof}

\section{Hilbert spaces of signed measures, and of distributions}

Let $K:U\times U\longrightarrow\mathbb{R}$ be a fixed positive definite
(p.d.) kernel, and let $\mathscr{H}\left(K\right)$ be the corresponding
RKHS. Suppose $U$ is a \emph{metric space}, and that $K$ is continuous.
(This is not a strict restriction since $K$ automatically induces
a metric $d_{K}$ on $U$ given by 
\begin{align*}
d_{K}\left(x,y\right) & =\left\Vert K\left(\cdot,x\right)-K\left(\cdot,y\right)\right\Vert _{\mathscr{H}\left(K\right)}\\
 & =\left(K\left(x,x\right)+K\left(y,y\right)-2K\left(x,y\right)\right)^{\frac{1}{2}},
\end{align*}
and 
\begin{equation}
\left|K\left(x_{1},y\right)-K\left(x_{2},y\right)\right|\leq d_{K}\left(x_{1},x_{2}\right)K\left(y,y\right)^{\frac{1}{2}},\;\forall x_{1},x_{2},y\in U.)\label{eq:hm1}
\end{equation}

Introduce the Dirac delta measures $\left\{ \delta_{x}\right\} $,
$x\in U$, we get $\delta_{x}K\delta_{y}=K\left(x,y\right)$; or more
precisely, 
\begin{equation}
\int_{U}\int_{U}K\left(s,t\right)d\delta_{x}\left(s\right)d\delta_{y}\left(t\right)=K\left(x,y\right).\label{eq:hm2}
\end{equation}
If $n\in\mathbb{N}$, $\alpha_{i}\in\mathbb{R}$, $x_{i}\in U$, $1\leq i\leq n$,
set 
\begin{equation}
\xi:=\sum\nolimits _{i=1}^{n}\alpha_{i}\delta_{x_{i}},\label{eq:hm3}
\end{equation}
and we get 
\begin{equation}
\xi K\xi=\sum\nolimits _{i}\sum\nolimits _{j}\alpha_{i}\alpha_{j}K\left(x_{i},x_{j}\right)=\left\Vert \sum\nolimits _{i}\alpha_{i}K\left(\cdot,x_{i}\right)\right\Vert _{\mathscr{H}\left(K\right)}^{2}.\label{eq:hm4}
\end{equation}
Hence, if we complete the measures from (\ref{eq:hm3}) with respect
to (\ref{eq:hm4}), we arrive at a Hilbert space $\mathscr{L}\left(K\right)$
consisting of signed measures, or of more general linear functionals,
e.g., distributions.
\begin{prop}
\label{prop:hm1}Let $K$, $U$, $\mathscr{H}\left(K\right)$ and
$\mathscr{L}\left(K\right)$ be specified as above; then 
\begin{equation}
J\left(\sum\nolimits _{i}\alpha_{i}\delta_{i}\right):=\sum\nolimits _{i}\alpha_{i}K\left(\cdot,x_{i}\right)\label{eq:hm5}
\end{equation}
defines an isometry of $\mathscr{L}\left(K\right)$ onto $\mathscr{H}\left(K\right)$
. 
\end{prop}

\begin{proof}
This follows from the definitions. In particular, for $\xi\in\mathscr{L}\left(K\right)$,
we have, by (\ref{eq:hm4})
\begin{equation}
\xi K\xi=\left\Vert J\xi\right\Vert _{\mathscr{H}\left(K\right)}^{2}.\label{eq:hm6}
\end{equation}
\end{proof}
\begin{cor}
Let $K$, $U$, $\mathscr{H}\left(K\right)$ and $\mathscr{L}\left(K\right)$
be as stated in \propref{hm1}, and let $\left(S,\mathscr{B},\mu\right)$
be a $\sigma$-finite measure space such that $L^{2}\left(\mu\right)$
is a feature space; see \defref{p1}. Then a signed measure $\xi$
is in $\mathscr{L}\left(K\right)$ if and only if 
\begin{equation}
\int_{S}\left|\int_{U}r_{x}\left(s\right)d\xi\left(x\right)\right|^{2}d\mu\left(s\right)<\infty;\label{eq:hm7}
\end{equation}
and in this case $\xi K\xi=$ the RHS in (\ref{eq:hm7}).
\end{cor}

\begin{proof}
Since $\mathscr{L}\left(K\right)\xrightarrow[\;\simeq\;]{\;J\;}\mathscr{H}\left(K\right)$
and $K\left(\cdot,x\right)\longrightarrow r_{x}$ extends by limiting
and closure to an isometry $\mathscr{H}\left(K\right)\longrightarrow L^{2}\left(\mu\right)$,
by \propref{tr1}, the following computation is valid when $\xi\in\mathscr{L}\left(K\right)$,
and vise versa: 
\begin{align*}
\int_{U}\int_{U}K\left(x,y\right)d\xi\left(x\right)d\xi\left(y\right) & =\int_{U}\int_{U}\int_{S}r_{x}\left(s\right)r_{y}\left(s\right)d\mu\left(s\right)d\xi\left(x\right)d\xi\left(y\right)\\
 & =\int_{S}\left|\int_{U}r_{x}\left(s\right)d\xi\left(x\right)\right|^{2}d\mu\left(s\right)=\text{RHS}_{\left(\ref{eq:hm7}\right)}.
\end{align*}
 
\end{proof}
\begin{rem}
If $U\times U\xrightarrow{\;K\;}\mathbb{R}\left(\text{or \ensuremath{\mathbb{C}}}\right)$
is $C^{\infty}$, or analytic for suitable choices of $U$and $K$,
then we may have distribution solutions $\xi$ to $\xi K\xi<\infty$.
A simple example illustrating this is $U=\left(-1,1\right)=$ the
interval, and 
\[
K\left(x,y\right)=\frac{1}{1-xy}.
\]
For $n\in\mathbb{N}$, let $\xi=\delta_{0}^{\left(n\right)}=$ the
$n^{th}$ derivative of the Dirac measure at $x=0$. Then 
\[
\int_{U}K\left(x,y\right)d\xi\left(y\right)=\frac{n!x^{n}}{\left(1-xy\right)^{n+1}}\Big|_{y=0}=n!x^{n};
\]
and
\[
\int_{U}\int_{U}K\left(x,y\right)d\xi\left(x\right)d\xi\left(y\right)=\left(n!\right)^{2}.
\]
In fact, 
\[
\delta_{0}^{\left(n\right)}K\delta_{0}^{\left(m\right)}=\delta_{n,m}\left(n!\right)^{2},\;\forall n,m\in\mathbb{N};
\]
so the system $\{\delta_{0}^{\left(n\right)}\}_{n\in\left\{ 0\right\} \cup\mathbb{N}}$
is orthogonal and total in $\mathscr{L}\left(K\right)$. If $x\in U\backslash\left\{ 0\right\} $,
then 
\[
\delta_{x}=\sum\nolimits _{n=0}^{\infty}\frac{x^{n}}{n!}\delta_{0}^{\left(n\right)},
\]
as an identity for compactly supported distributions. 
\end{rem}

\begin{acknowledgement*}
The co-authors thank the following colleagues for helpful and enlightening
discussions: Professors Daniel Alpay, Sergii Bezuglyi, Ilwoo Cho,
Myung-Sin Song, Wayne Polyzou, and members in the Math Physics seminar
at The University of Iowa.
\end{acknowledgement*}
\bibliographystyle{amsalpha}
\bibliography{ref}

\end{document}